\documentclass[11pt]{amsart}
\usepackage{amsfonts, bm}
\usepackage{mathrsfs}
\allowdisplaybreaks
\usepackage{hyperref}

\makeatletter

\renewcommand{\tocsection}[3]{%
  \indentlabel{\@ifnotempty{#2}{\bfseries\ignorespaces#1 #2\quad}}\bfseries#3}

\usepackage{amsmath}
\usepackage{amssymb}
\usepackage{multicol}
\usepackage{stmaryrd}
\usepackage{cite}
\usepackage{epsfig}
\usepackage{color}
\usepackage{graphics}
\usepackage{graphicx}
\usepackage{multicol,graphics}
\newcommand\bes{\begin{eqnarray}}
\newcommand\ees{\end{eqnarray}}
\newcommand\R{\mathbb R}
\newtheorem{theorem}{Theorem}[section]
\newtheorem{lemma}[theorem]{Lemma}
\newtheorem{corollary}[theorem]{Corollary}

\newtheorem{remark}[theorem]{Remark}

\numberwithin{equation}{section}

\theoremstyle{plain}
\newtheorem*{theorem*}{Theorem A}

\newcommand\bess{\begin{eqnarray*}}
	\newcommand\eess{\end{eqnarray*}}
\newcommand{\lf}{\left}
\newcommand{\rr}{\right}
\newcommand{\dd}{\displaystyle}

\newcommand{\td}{\tilde}

\newcommand\yy{\infty}

\newcommand{\ol}{\overline}
\newcommand{\rd}{{\rm d}}

\begin{document}

\title[Dynamics of an invading competitor]{Long-time dynamics of a competition model with\\  nonlocal diffusion and free boundaries: \\ Chances of  successful invasion}
\author[Y. Du, W. Ni, L. Shi ]{Yihong Du$^\dag$, Wenjie Ni$^\dag$ and Linfei Shi$^{\ddag}$}
\thanks{\hspace{-.5cm}
$^\dag$ School of Science and Technology, University of New England, Armidale, NSW 2351, Australia.
\\
$^{\ddag}$ School of Mathematics and Statistics, Beijing Institute of Technology, Beijing 100081, China.
\\
\mbox{\ \  Emails:} ydu@une.edu.au (Y. Du),\ wni2@une.edu.au (W. Ni),\ 3120205702@bit.edu.cn (L. Shi)}

\date{\today}

\begin{abstract}

This is a continuation of our work \cite{dns-part1} to investigate the long-time dynamics of a two species competition model of Lotka-Volterra type with nonlocal diffusions, where the territory (represented by the real line $\R$) of a native species with density $v(t,x)$,  is invaded by a competitor with density $u(t,x)$, via two fronts, $x=g(t)$ on the left and $x=h(t)$ on the right. So the population range of $u$ is the evolving interval $[g(t), h(t)]$ and the reaction-diffusion equation for $u$ has two free boundaries, with $g(t)$  decreasing in $t$ and $h(t)$ increasing in $t$. Let $h_\infty:=h(\infty)\leq \infty$ and $g_\infty:=g(\infty)\geq -\infty$.
In  \cite{dns-part1}, we obtained detailed descriptions of the long-time dynamics of the model according to whether $h_\infty-g_\infty$ is $\infty$ or finite.
In the latter case, we demonstrated in what sense the invader $u$ vanishes in the long run and $v$ survives the invasion, while in the former case, we obtained a rather satisfactory description of the long-time asymptotic limits of $u(t,x)$ and $v(t,x)$ when the parameter $k$ in the model is less than 1. In the current paper,  we obtain sharp criteria  to distinguish the case $h_\infty-g_\infty=\infty$ from the case $h_\infty-g_\infty$ is  finite. Moreover, for the case $k\geq 1$ and $u$ is a weak competitor, we obtain biologically meaningful conditions that guarantee the vanishing of the invader $u$, and reveal chances  for $u$ to invade successfully. In particular, we demonstrate that both $h_\infty=\infty=-g_\infty$ and $h_\infty=\infty$ but $g_\infty$ is finite are possible; the latter seems to be the first example for this kind of population models, with either local or nonlocal  diffusion.

\bigskip

\medskip


\end{abstract}

\maketitle

\section{Introduction}
We continue our work  \cite{dns-part1} on  the following Lotka-Volterra type competition model with nonlocal diffusion and free boundaries
\begin{align}\label{KnK1.2}
	\begin{cases}
	\dd	u_t = d_1\int_{g(t)}^{h(t)}J_1(x - y)u(t, y)\rd y - d_1u + u(1 - u - kv),   &
		t > 0, ~g(t) < x < h(t),\\[4mm]
\dd		v_t = d_2\int_\mathbb{R}J_2(x - y)v(t, y)\rd y - d_2v + \gamma v(1 - v - hu),   &
		t > 0,~x\in\mathbb{R},\\[3mm]
		u(t, x) = 0,  &
		t > 0,\ x \not\in (g(t), h(t)),
		\\[2mm]
	\dd	h^{\prime} (t) = \mu\int_{g(t)}^{h(t)}\int_{h(t)}^{\infty}J_1(x - y)u(t, x)\rd y\rd x,  &
		t > 0,\\[4mm]
	\dd	g^{\prime} (t) = -{\mu}\int_{g(t)}^{h(t)}\int_{-\infty}^{g(t)}J_1(x - y)u(t, x)\rd y\rd x,  &
		t > 0,\\[3mm]	
		h(0) = -g(0) = h_0 > 0,\ u(0, x) = u_0(x), \
		 v(0, x) = v_0(x),
		&
		x \in\mathbb{R},
	\end{cases}
\end{align}
where $d_1, d_2, h, k, \gamma, \mu$ are given positive constants, and the initial functions satisfy 
\begin{align}\label{KnK1.3}\begin{cases}
		u_{0} \in C(\R), ~u_{0}(x) = 0 ~ {\rm for}~|x|\geq h_0, ~ u_{0}(x) > 0  ~ {\rm for}~|x|< h_0,\\
		v_0\in C_b(\mathbb{R}), ~ v_{0}(x) \geq 0,\ v_0(x)\not\equiv 0 {\rm ~in~} \mathbb{R},
		\end{cases}
\end{align}
where $C_b(\mathbb{R})$ is the space of continuous and bounded functions in $\mathbb{R}.$ 

We assume that the kernel functions $J_1$ and $ J_2$ satisfy\smallskip

\noindent$(\mathbf{J}):$  \ \ $\displaystyle J_i\in C_b(\mathbb{R})$, $J_i(x)=J_i(-x) \geq 0,$ $J_i(0) > 0,$ $\displaystyle\int_{\mathbb{R}}J_i(x)dx = 1$ for $i=1,2$.

\medskip

Under these assumptions,  it is known that system \eqref{KnK1.2} has a unique solution $(u, v, g, h)$ defined for all $t > 0$ (see \cite{CLWZ}). Moreover,
 \[
 \mbox{ $g_\yy:=\lim_{t\to\yy}g (t)\in [-\infty, -h_0)$ and $h_\yy:=\lim_{t\to\yy}h(t)\in (h_0, \infty]$}
 \]
  always exist. 
 
 In \cite{dns-part1},  the long-time dynamics of \eqref{KnK1.2} are described according to the following two cases:
\[
{\bf (a)\!:} \  h_\infty-g_\infty<\infty,\ \ {\bf (b)\!:} \ h_\infty-g_\infty=\infty.
\]
For case {\bf (a)}, we have proved  the following result. \medskip

\noindent {\bf Theorem A.} {\it  Assume that $(\mathbf{J})$  holds and $(u, v, g, h)$ is the unique solution of \eqref{KnK1.2}.  If  $h_\infty - g_\infty < \infty$, then necessarily
	\begin{align}\label{3.3}
	d_1>1-k \mbox{ and }  \lambda_p(\mathcal{L}_{(g_\infty, h_\infty)})  \leq k-1.
	\end{align}
	moreover
	\begin{equation}\label{vanishing}\begin{cases}
	\dd\lim_{t\to\infty}\int_{\R}u(t,x)dx=0,\\[2mm]
	\dd\lim_{t\to\infty}\int_L^L|v(t,x)-1)|dx=0 \mbox{ for every } L>0,\\[2mm]
	\dd\lim_{t\to\infty} v(t,x)=1 \mbox{ locally uniformly for } x\in\R\setminus (g_\infty, h_\infty).
	\end{cases}
	\end{equation}
}

Whether \eqref{vanishing} in Theorem A can be strengthened to
\begin{equation}\label{3.4}
		\begin{aligned}
			&\lim\limits_{t \rightarrow \infty}\max\limits_{x\in[g(t), h(t)]} u(t, x) = 0\ \mbox{ and } \  \lim\limits_{t \rightarrow \infty}v(t, x) = 1 {\rm ~locally ~uniformly ~ for~} x\in\mathbb{R}
		\end{aligned}
	\end{equation}
 was partially answered in \cite{dns-part1} (see Theorem 1.2 there).
\medskip

For case {\bf (b)}, we have obtained in \cite{dns-part1} the following conclusion. 
\medskip

\noindent {\bf Theorem B.} {\it
Assume that $(\mathbf{J})$  holds and $(u, v, g, h)$ is the unique solution of \eqref{KnK1.2}.
	If  $h_\yy-g_\yy=\yy$ and $k<1$, then   $h_\infty = \infty,$ $g_\infty = -\infty$
	and
	\[
	\lim_{t\to\infty}(u(t,x), v(t,x))=\begin{cases} (1,0) &\mbox{ if } h\geq 1,\\
	(\frac{1-k}{1-hk},\frac{1-h}{1-hk}) &\mbox{ if } h<1,
	\end{cases}
	\]
	where the convergence is locally uniform for $x\in\R$.
}
\medskip

Let us recall that, in \eqref{3.3}, $\lambda_p(\mathcal{L}_{(g_\infty, h_\infty)})$ denotes the principal eigenvalue  of   the
 following eigenvalue problem
\begin{equation}\label{eigen}
	\begin{aligned}
\lambda\varphi=		\mathcal{L}_{\Omega} [\varphi](x): = d_1\left[\int_{\Omega}J_1(x - y)\varphi(y)\rd y - \varphi(x)\right] , ~\varphi\in C(\overline{\Omega}),
	\end{aligned}
\end{equation}
with $\Omega=(g_\infty, h_\infty)$. It is well known that,
 under our assumption ({\bf J}), for any finite interval $\Omega$,
\eqref{eigen}  has a unique principal eigenvalue $\lambda=\lambda_p( \mathcal{L}_{\Omega})$ associated with a positive eigenfunction $\varphi$ (e.g. \cite{BCV, C, LCW}), and
 it  has the following properties:\medskip

\noindent {\bf Proposition C.} (\!\!\cite[Proposition 3.4]{Cao-2019-JFA}){\it \
Assume that $l > 0$, and $J_1$ satisfies $(\mathbf{J})$.  Then 
	
	{\rm (i)} $\lambda_p(\mathcal{L}_{(a, a+l)})=\lambda_p(\mathcal{L}_{(0, l)})$  for all  $ a\in\R$,
	
	{\rm(ii)} $\lambda_p(\mathcal{L}_{(0, l)})$ is  strictly increasing and continuous in $l$,
	
	{\rm(iii)} $\lim\limits_{l \rightarrow \infty}\lambda_p(\mathcal{L}_{(0, l)}) = 0$,
	
	{\rm(iv)} $\lim\limits_{l \rightarrow 0}\lambda_p(\mathcal{L}_{(0, l)}) =- d_1$.}

\medskip

\noindent
Therefore, for every  $\sigma\in (0, d_1)$,  there exists a unique $l_\sigma> 0$   such that 
\[
	\lambda_p\left(\mathcal{L}_{(0,l_\sigma)}\right)=-\sigma.
\]
\bigskip

We are now ready to  describe our results in this paper. Firstly, we  examine exactly when $h_\infty-g_\infty<\infty$ and $h_\infty-g_\infty=\infty$, respectively, happens. Then we  focus on the situation that $u$ is a weak competitor ($k\geq 1>h$) and 
reveal some interesting phenomena; in particular, we will find conditions for $u$ to invade successfully, with $h_\infty=\infty,\ g_\infty=-\infty$, as well as with $h_\infty=\infty$ and $g_\infty$  finite.

By Theorem A, the fact $h_\infty-g_\infty<\infty$ implies that, $\int_{g(t)}^{h(t)}u(t,x)dx$,  the total population of $u$ at time $t$, converges to 0 as $t\to\infty$, so the invading competitor $u$ vanishes in the long run. We will call this the {\bf vanishing} (of $u$) case. 

The indentity $h_\infty-g_\infty=\infty$ means that the size of the population range of $u$ at time $t$, given by $h(t)-g(t)$, converges to $\infty$ as $t\to\infty$, and we will call this the {\bf spreading} (of $u$) case.  Theorem B gives a precise description for the population densities $u(t,x)$ and $v(t,x)$ in this case when $k<1$. We will demonstrate below that when $k\geq 1$, more complicated dynamics may  arise (see Theorems \ref{th1.2}, \ref{v-1}, \ref{th1.3} and \ref{th1.4}).
\medskip

To describe the criteria governing spreading and vanishing (of $u$), we will regard $\mu$ as a parameter in certain situations. 
	
\begin{theorem}\label{th1.1} Suppose that {\rm ({\bf J})} holds. Then the following conclusions are valid:
\begin{itemize}
	\item[{\rm (i)}]  If $k<1$ and $d_1 \leq  1 - k$, then  we always have $h_\infty-g_\infty=\infty$.
	
	\item[{\rm (ii)}]  If $k<1$, $d_1 > 1 - k$ and $2h_0\geq l_{1-k}$, then again $h_\infty-g_\infty=\infty$ always holds.
	
	\item[{\rm (iii)}]  If  $k<1$, $d_1 > 1 - k$ and $2h_0< l_{1-k}$, then there exists $\mu_*\in [0,\infty)$, depending on  $(u_0, v_0)$, such that
	$h_\infty-g_\infty=\infty$ exactly when $\mu>\mu_*$; moreover, 
	 $\mu_*>0$ if $1<d_1$ and $2h_0< l_1$.
	 \item[\rm (iv)] If $k\geq 1$, then there exists $\mu_*\in [0,\infty]$, depending on  $(u_0, v_0, h_0)$, such that
	$h_\infty-g_\infty=\infty$ exactly when $\mu>\mu_*$; moreover, $\mu_*\in (0, \infty]$ when $k>1>h$.
\end{itemize}
\end{theorem}

The above results indicate that $1-k$ serves as a critical diffusion rate for $u$: If its diffusion rate $d_1\leq 1-k$, then successful invasion is guaranteed regardless of the choice of the initial data $(u_0, v_0, h_0)$ as long as they are admissible, namely satisfying \eqref{KnK1.3}; if $d_1>1-k>0$, then the size of the initial population range $2h_0$ becomes crucial and $l_{1-k}$ is a critical value for this initial size, and successful invasion is guaranteed when $2h_0\geq l_{1-k}$. When $k<1$, if both the diffusion rate and the size of the initial population range of $u$ are below their respective critical values, then the value of the parameter $\mu$ becomes important, and there exists a critical value $\mu_*$, depending on $(u_0, v_0)$, such that the invasion is successful if and only if $\mu>\mu_*$. Similarly, when $k\geq 1$, there exists a critical value $\mu_*$, depending on $(u_0, v_0, h_0)$, such that the invasion is successful if and only if $\mu>\mu_*$.
\bigskip

Next we regard $\mu>0$ as a fixed given constant, and look for biologically meaningful sufficient conditions guaranteeing vanishing and spreading of $u$, respectively.
We will focus on the weak-strong competition case with  $u$ the weak competitor, namely
\[
k>1>h.
\]
In some of our results, $k=1$ is also allowed.

Similar to the corresponding local diffusion model considered in \cite{DuLin} (see Theorem 3.3 there), 
 the invasion of the weak competitor $u$ will definitely fail if  the native species $v$ is already well established at time $t=0$, namely 
\begin{equation}\label{v0->0}
\inf_{x\in\R} v_0(x)>0.
\end{equation}
Indeed, following the proof of  Theorem 3.3 in \cite{Zwy-2022-dcds-b}, one can easily show the following result:\medskip

\noindent
{\bf Proposition D.} 
{\it If $v_0$ satisfies \eqref{v0->0} and $k>1>h$, then $h_\infty-g_\infty<\infty$, and as $t\to\infty$, 
\[\begin{cases}\mbox{$u(t,x)\to 0$ uniformly for $x\in [g(t), h(t)]$,}\\
\mbox{$v(t,x)\to 1$  uniformly for $x\in\R$. }
\end{cases}\]
}
\smallskip

Another situation that  the invasion of $u$  always fails is when
 $u$ has the same dispersal strategy  and the same growth rate to the stronger native species $v$,  as described in the following result. 

\begin{theorem}\label{th1.2}
	If  $k\geq 1>h$,  $d_1=d_2=d$, $\gamma=1$, $J_1=J_2=J$ and
	\begin{align}\label{thin-tail}
	\dd 
	\int_0^\infty \!\! e^{\lambda x}J(x)dx <\yy\ \ \ {\rm for\ some }\ \lambda>0,
\end{align}
 then for any   $(u_0, v_0)$ satisfying \eqref{KnK1.3}, $h_\infty-g_\infty<\infty$ and as $t\to\infty$, 
 \[\begin{cases}\mbox{$u(t,x)\to 0$ uniformly for $x\in [g(t), h(t)]$,}\\
\mbox{$v(t,x)\to 1$ locally uniformly for $x\in\R$. }
\end{cases}\]
\end{theorem}

The assumption \eqref{thin-tail} guarantees that, in the absence of $u$, the species $v$ has a finite spreading speed, which is helpful for our proof, but we believe it is only a technical
condition. A kernel function satisfying \eqref{thin-tail} is known as a ``thin-tailed" kernel in the literature.
\medskip

Let use note that Proposition D and Theorem \ref{th1.2} are examples where $\mu_*=\infty$ in Theorem \ref{th1.1} (iv).

\medskip

The next result indicates that when $k>1>h$, the native species $v$ never vanishes.

\begin{theorem}\label{v-1}
If $k>1>h$, then for any $(u_0,v_0)$ satisfying \eqref{KnK1.3}, and any $L>0$, we have
\[\begin{cases}
\lim_{t\to\infty} u(t,x)=0 &\mbox{ uniformly for } x\in [-L, L],\\
\lim_{t\to\infty} v(t,x)=1 &\mbox{ uniformly for } x\in [-L, L].
\end{cases}
\]
\end{theorem}

If the native species $v$ is not  well established at time $t=0$, for example, if $v_0$ is compactly supported (see Remark 1.5  (i) below for other natural choices of $v_0$), we will show that there are indeed chances for the weak competitor $u$ to invade successfully and establish itself in an increasing band behind the invasion front which goes to infinity as $t\to\infty$. Moreover, the invasion can succeed in both fronts (i.e., $h_\infty=\infty,\ g_\infty=-\infty$), or just in one front ($h_\infty=\infty$, $g_\infty$ is finite). 

To achieve these,  we assume that $J_1$ and $J_2$ both have compact support (for technical reasons), and the dispersal strategy of $u$ makes it a faster spreader than $v$, in the sense explained in the next several paragraphs, based on the notion of asymptotic spreading speed described in Proposition E and in \cite{DLZ-2021-JMPA}.

For a kernel function $J$ satisfying ({\bf J}) and \eqref{thin-tail}, consider the Cauchy problem of the  logistic 
equation
\begin{equation}\label{w}\begin{cases}
\dd w_t=d\big[\int_{\R}J(x-y)w(t,y)dy-w\big]+ aw-b w^2, & t>0,\ x\in\R,\\
w(0,x)=w_0(x), & x\in\R,
\end{cases}
\end{equation}
where
 $a, b$ and $d$ are positive constants,   $w_0(x)\geq 0$ is a continuous function with nonempty compact support.

It is well known (see, for example, \cite{CC, CDM, W, Y}) that the following results hold for \eqref{w}.

\medskip

\noindent {\bf Proposition E.} {\it The asymptotic spreading speed determined by \eqref{w} is given by
\[
c^*:=\min_{\lambda>0}\frac 1\lambda \left(d\int_{\R}J(x) e^{\lambda x}dx-d+a\right),
\]
namely,
 for any small $\epsilon>0$,
 \[
 \begin{cases}
 \lim_{t\to\infty}\max_{|x|\leq (c^*-\epsilon)t}|w(t,x)-\frac{a}b|=0,\\
 \lim_{t\to\infty}\max_{|x|\geq (c^*+\epsilon)t}w(t,x)=0.
 \end{cases}
 \]
 Moreover, for any $c\geq c^*$, there exists  a monotone function  $\phi=\phi_c\in  C^1(\R)$, called a traveling wave of \eqref{w} with speed $c$,  satisfying
	\begin{equation*}
		\begin{cases}
			\displaystyle d \int_{-\infty}^0 J(x-y) \phi(y) dy - d \phi+ c\phi' + a\phi -b\phi^2=0, &  -\infty < x< \infty,\\
			\displaystyle \phi(-\infty) = 1,\ \ \phi(\infty) =0.
					\end{cases}
	\end{equation*}
	Such a traveling wave is unique up to a translation of $x$, and no such traveling wave exists for speed $c<c^*$.
 }
 \medskip
 
Since
\[
\int_{\R}J(x) e^{\lambda x}dx=\int_0^\infty J(x)(e^{\lambda x}+e^{-\lambda x})dx>\int_{\R} J(x)dx=1,
\]
we see that
\begin{equation}\label{c*}
c^*>d \min_{\lambda>0}\frac 1\lambda \left(\int_{\R}J(x) e^{\lambda x}dx-1\right)\to\infty \mbox{ as } d\to\infty.
\end{equation}
\smallskip

For our competition system \eqref{KnK1.2}, in the absence of $u$, clearly $v$ satisfies \eqref{w} with  $(d, J, a,b, w_0)=(d_2, J_2, \gamma, \gamma, v_0)$.
So the asymptotic spreading speed of $v$ (in the absence of $u$, and with $v_0$ compactly supported) is
\[
C_2:=\min_{\lambda>0}\frac 1\lambda \left(d_2\int_{\R}J_2(x) e^{\lambda x}dx-d_2+\gamma\right).
\]
By \cite{DLZ-2021-JMPA}, in the absence of $v$, the asymptotic spreading speed of $u$, denoted by $c_1=c_1(\mu)$, satisfies
\[
0<c_1(\mu)<C_1, \ \lim_{\mu\to\infty} c_1(\mu)=C_1:=\min_{\lambda>0}\frac 1\lambda \left(d_1\int_{\R}J_1(x) e^{\lambda x}dx-d_1+1\right).
\]

We will say that {\bf ${\bf u}$ is a faster spreader than ${\bf v}$} if $c_1>C_2$, which is guaranteed, for instance, if $C_1>C_2$ and $\mu$ is sufficiently large. By \eqref{c*} we see that $C_1>C_2$ is guaranteed if $d_1$ is  sufficiently large when the other parameters are fixed.

\begin{theorem}\label{th1.3}
	Suppose that $v_0$ is compactly supported,   $k\geq 1>h$,  $J_1$ and $J_2$ have compact support  and $c_1>C_2$. 	 Then one can find initial functions $u_0$ for $u$ such that $(g_\infty, h_\yy)=(-\yy,\yy)$. 
		 \end{theorem}

 \begin{theorem}\label{th1.4}
	In Theorem \ref{th1.3}, suppose additionally $k(1-h)>1$, then  there exist initial functions $u_0$ 
	  such that $h_\yy=\infty$ and $ g_\yy$ is finite. 
	  	 \end{theorem}

\begin{remark} In Theorems \ref{th1.3} and \ref{th1.4},  the following additional results hold:
\begin{itemize}
\item[(i)] The assumption that $v_0$ has compact support can be replaced by some more natural assumptions. For example, if $V(t,x)$ stands for the solution of \eqref{w} with $(d, J, a,b)=(d_2, J_2, \gamma, \gamma)$, where the initial function $w_0$ is compactly supported, then it is easy to check that the conclusions in Theorems \ref{th1.3} and \ref{th1.4} remain valid for any  $v_0$ satisfying 
\[\mbox{$ 0<  v_0(x)\leq V(t_0, x)$ for some $t_0>0$ and all $x\in\R$.}
\]
\item[(ii)]
The behaviours of the density functions $u(t,x)$ and $v(t, x)$ in Theorems \ref{th1.3} and \ref{th1.4} are described in Remarks \ref{rm4.1} and \ref{rm:4.2} later in the paper, immediately after the respective proofs of these theorems. 
\end{itemize}
\end{remark}
\medskip

The rest of the paper is organised as follows. In Section 2, we prove Theorem \ref{th1.2}. Section 3 is devoted to the proof of Theorems \ref{th1.3} and \ref{th1.4}.
The proof of Theorems \ref{th1.1} and \ref{v-1} are given in Section 4. The final Section 5 is an appendix, where we list several comparison principles used in this paper, whose proofs are not included, as they are simple variations of existing ones.

\section{Proof of Theorem \ref{th1.2}}

 It is clear that $v(t,x)>0$ for all $t>0$ and $x\in \R$. Therefore, for fixed $t_0>0$, there exists $\alpha_0>0$ such that 
	\begin{align*}
		u(t_0,x)\leq \alpha_0v(t_0,x)\ \mbox{ for }\ x\in \R,
	\end{align*}
	where we have used the assumption $u(t_0,x)=0$ for all $x\in \R\backslash [g(t_0),h(t_0)]$.
	
	{\bf Step 1}. We show $u(t,x)\leq  \alpha_0v(t,x)$ for all $t\geq t_0$ and $x\in \R$.
	
	 Denote $\td v(t,x):=\alpha_0 v(t,x)$. Then $(u,\td v)$ satisfies
	\begin{align*}
		\begin{cases}\dd	u_t = d\int_{g (t)}^{h(t)}J(x - y)u(t, y)\rd y - du + u(1 - u -k \td v/\alpha_0 ),   &
		t > 0, ~g(t) < x < h(t),\\[3mm]
				\dd	\td 	v_t = d\int_\mathbb{R}J(x - y)\td  v(t, y)\rd y - d\td  v +  \td v(1 - \td v/\alpha_0- hu),   &
		t > 0,~x\in\mathbb{R},
		\end{cases}
	\end{align*}
Let $w:=\td v-u$. Then due to $h<1\leq k$, the function $ w(t,x)$ satisfies 
\begin{align*}
	\dd	w_t = \ & d\int_{g(t)}^{h(t)}J(x - y)w(t, y)\rd y - dw +w\lf (1-\frac{\td v}{\alpha_0} -\frac{\alpha_0 h+1-k}{\alpha_0}u\rr)+\lf(1-h+\frac{k-1}{\alpha_0}\rr)u^2\\
	\geq\ & d\int_{g(t)}^{h(t)}J(x - y)w(t, y)\rd y - dw +w\lf (1-\frac{\td v}{\alpha_0} -\frac{\alpha_0 h+1-k}{\alpha_0}u\rr) \mbox{ for } \ t\geq t_0,\ x\in [g (t),h(t)].
\end{align*}
Clearly $w(t,x)\geq 0$ for $x=g (t)$ and $h(t)$,  and $w(t_0,x)\geq 0$ for $x\in [g(t_0),h(t_0)]$. It then follows from the comparison principle Lemma \ref{lemmacomp1} that 
	\begin{align*}
		w(t,x)\geq 0, \mbox{ i.e.}, \ u(t,x)\leq  \alpha_0v(t,x) \mbox{ for }\ t\geq t_0,\ x\in \R.
	\end{align*}

{\bf Step 2}. Estimates of $u$ and $v$ leading to the desired conclusion.

Using the result in Step 1, we see that $(u,v)$ satisfies
\begin{align*}
	\begin{cases}
		\dd	u_t \leq  d\int_{g(t)}^{h(t)}J(x - y)u(t, y)\rd y - du + u[1 - (1+k/\alpha_0 )u],  &
		t \geq t_0, ~g(t) < x < h(t),\\[4mm]
		\dd	 	v_t \geq  d\int_\mathbb{R}J(x - y)  v(t, y)\rd y - d  v +   v[1 - (1+ h\alpha_0)v],  &
		t \geq t_0,~x\in\mathbb{R},
	\end{cases}
\end{align*}

This allows us to compare $u$ and $v$ with $\hat U$, $U$ and $ V$  defined below:

\begin{itemize}
\item $(\hat U, \hat h, \hat g)$ is the solution of 
\begin{align}\label{hat-U}
	\begin{cases}
		\dd	\hat U_t = d \int_{\hat g (t)}^{\hat h(t)}J(x - y) \hat U(t, y)\rd y - d \hat U + \hat U[1-(1+k/\alpha_0)\hat U],   &
		t >t_0,\ ~\hat g (t) < x < \hat h(t),\\[4mm]
		\dd	\hat h^{\prime} (t) = \mu\int_{\hat g (t)}^{\hat h(t)}\int_{\hat h(t)}^{\infty}J(x - y)\hat U(t, x)\rd y\rd x,  &
		t > t_0,\\[4mm]
		\dd	\hat g^{\prime} (t) = -{\mu}\int_{\hat g (t)}^{\hat h(t)}\int_{-\infty}^{\hat g (t)}J(x - y)\hat U(t, x)\rd y\rd x,  &
		t > t_0,\\	
			\hat U(t, x) = 0,  &
		t \geq t_0,\  x \not\in (\hat g (t), \hat h(t)),\\
		\hat h(t_0) =h(t_0),\ \hat g (t_0)= g (t_0),\ \hat U(t_0, x) = u(t_0,x), &
		g (t_0) \leq x \leq h(t_0).
	\end{cases}
\end{align}

\item $U$ is the solution of
\begin{align}\label{U}
	\begin{cases}
	\dd		U_t = d\int_\mathbb{R}J(x - y)U(t, y)\rd y - d U +  U[1-(1+k/\alpha_0)U],   &
	t > t_0,~x\in\mathbb{R},\\[4mm]
	U(t_0, x) = u(t_0, x),	&
	t \geq 0,x \in\mathbb{R}.
\end{cases}
\end{align}

\item  $V$ be the solution of
\begin{align}\label{V}
	\begin{cases}
	\dd		V_t = d\int_\mathbb{R}J(x - y)V(t, y)\rd y - d V + V[1 - (1+ h\alpha_0)V],   &
	t > t_0,~x\in\mathbb{R},\\[4mm]
	V(t_0, x) = v(t_0, x),	& x \in\mathbb{R}.
\end{cases}
\end{align}
\end{itemize}

 By \cite[Theorem 1.1 and Proposition 1.3]{DLZ-2021-JMPA},  there exist  $C_1>0$ and $c_1=c_1(\mu)\in(0, C_1)$ 
such that, for any small $\epsilon>0$,
\begin{align}\label{speed1}
	\begin{cases}
		\dd \lim_{t\to \yy} \frac{\hat g (t)}{t}=\lim_{t\to \yy} \frac{\hat h(t)}{t}=c_1,\\
		\dd 	\lim_{t\to \yy}\max_{|x|\leq (C_1-\epsilon)t} |U(t,x)-\frac 1{1+k/\alpha_0}|=0,\\
	\dd 	\lim_{t\to \yy}\max_{|x|\leq (C_1-\epsilon)t} |V(t,x)-\frac 1{1+h\alpha_0}|=0.
	\end{cases}
\end{align} 

Using suitable versions of the comparison principle we have
\[
 u(t,x)\leq U(t,x),\ v(t,x)\geq V(t,x) \mbox{ for } t\geq t_0,\ x\in \R,
\]
and
\[
u(t,x)\leq \hat U(t,x), \ [ g (t), h(t)]\subset [\hat g (t), \hat h(t)] \mbox{ for } t\geq t_0,\ x\in [g (t), h(t)].
\]
Hence, by \eqref{speed1}, for $\td C:=(c_1+C_1)/2$ and all large $t$,  say $t\geq t_1$,  we have
\begin{align}\label{4.12a}
	\begin{cases}
	\dd u(t,x)\leq \frac{1}{1+k/\alpha_0}+o_t(1) & \mbox{ for }\ x\in [g(t),h(t)]\subset [-\td Ct,\td C t], \\
		\dd v(t,x)\geq  \frac{1}{1+\alpha_0h}-o_t(1)& \mbox{ for }\ x\in  [-\td C t,\td C t],
	\end{cases}
\end{align}
 where  $o_t(1)$ denotes a generic constant satisfying $o_t(1)\to 0$ as $t\to\infty$.

Case 1: 
 $\beta_0:=1 -\frac{k}{1+\alpha_0h}<0$. In this case,    by  \eqref{4.12a}, for all large $t$, say $t\geq t_2>t_1$,
\begin{align*}
	\dd	u_t \leq  d_1\int_{g(t)}^{h(t)}J(x - y)u(t, y)\rd y - d_1u + \frac{\beta_0}{2}u  \ \mbox{ for } \ 
	 ~g(t) \le x \le h(t),
\end{align*}
which implies, by comparison with the corresponding ODE, that
\[
\mbox{$u(t,x)\leq e^{(t-t_2)\beta_0/2} \max_{x\in [g(t_2),h(t_2)]} u(t_2,x)$ for $t\geq t_2$ and $x\in [g(t),h(t)]$.}
\]
 Thus, $u(t,x)\to 0$ as $t\to\infty$ uniformly for $x\in [g(t), h(t)]$, and by the free boundary equation for $h'(t)$ and the estimate
\begin{align*}
	\infty>J_{\max}:=\int_{0}^{\yy} J(y)y \rd y=\int_{-\yy}^{0} \int_{0}^{\yy} J(x-y) \rd y \rd x\geq  \int_{g(t)}^{h(t)}\int_{h(t)}^{\yy} J(x-y)\rd y \rd x,
\end{align*}
we obtain, for some constant $C>0$,
\begin{align*}
	h(t)&=h(t_2)+\int_{t_2}^{t}h'(s) \rd s\leq h(t_2)+C\mu J_{\max}  \int_{t_2}^{t} e^{\beta_0 (s-t_2)/2}\rd s \\
	&\leq h(t_2)+\frac{2C\mu J_{\max} }{|\beta_0|}<\yy\ \mbox{ for }\ \ t\geq t_2.
\end{align*}
So $h_\infty<\infty$. Similarly we can show $g_\infty>-\infty$. Moreover, using the fact $u(t,x)\to 0$ as $t\to\infty$ uniformly for $x\in [g(t), h(t)]$, and the comparison principle, it can be easily shown that $v(t,x)\to 1$ locally uniformly for $x\in\R$ as $t\to\infty$.

So in the  case $\beta_0<0$, the desired conclusions hold.

Case 2:  $\beta_0=1 -\frac{k}{1/\alpha_0+h}\geq 0$.

In this case we can repeat the method  leading to \eqref{4.12a} in the following way. Firstly since $u(t,x)\equiv 0$ for $x\not\in [g(t),h(t)]$, we  see from \eqref{4.12a} that 
\begin{align*}
		u(t,x)\leq [\alpha_1+o_t(1)]v(t,x) \mbox{ for all large $t$ and } \ x\in \R, 
\end{align*}
with 
\begin{align*}
	\alpha_1:=\frac{\alpha_0+\alpha_0^2h}{\alpha_0+k}=\alpha_0-\frac{\alpha_0(k-1)+\alpha_0^2(1-h)}{\alpha_0+k}.
\end{align*}
Clearly,  $\alpha_1<\alpha_0$. We are now in a position to repeat the earlier argument  to obtain an analogue of \eqref{4.12a}, namely
\begin{align*}\begin{cases}
	u(t,x)\leq \frac{1}{1+k/\alpha_1}+o_t(1) & \mbox{ for all large $t$ and } \ x\in [g(t),h(t)]\subset [-\td Ct,\td C t],\\
	v(t,x)\geq  \frac{1}{1+\alpha_1h}-o_t(1) & \mbox{ for all large $t$ and }\ x\in  [-\td Ct,\td C t].
	\end{cases}
\end{align*}
Then, we  estimate $u$ and $v$ according to whether $\beta_1:=1 -\frac{k}{1+\alpha_1h}<0$ or $\beta_1\geq 0$.

If $\beta_1<0$, then we can deduce the desired conclusions by similar arguments as in Case 1 above.

If $\beta_1\geq 0$, then we  analogously define
\begin{align*}
	\alpha_2:=\frac{\alpha_1+\alpha_1^2h}{\alpha_1+k}=\alpha_1-\frac{\alpha_1(k-1)+\alpha_1^2(1-h)}{\alpha_1+k},
\end{align*} 
and obtain
\[
u(t,x)\leq [\alpha_2+o_t(1)]v(t,x) \mbox{ for all large $t$ and } \ x\in \R.
\]

Following this procedure, we obtain two decreasing sequences $\{\alpha_m\}$ and $\{\beta_m\}$ given by
\begin{align*}
	\alpha_m:=\frac{\alpha_{m-1}+\alpha_{m-1}^2h}{\alpha_{m-1}+k},\  \beta_m:=1 -\frac{k}{(1+\alpha_{m-1}h)},\ m=1,2,...
\end{align*}
We claim that  there is an integer $m_0\geq 2$  such that $\beta_{m_0-1}\geq 0$ and
\begin{align*}
	\beta_{m_0}=1 -\frac{k}{1+\alpha_{m_0}h}<0\ \ {\rm or\ equivalently}\ \ \alpha_{m_0}<\frac{k-1}{h}.
\end{align*}
 Note that the cases of $m_0=0$ and $m_0=1$ have already been considered above.   
If  $\alpha_m\geq \alpha_*:=(k-1)/h$ for all $m\geq 1$, then
\begin{align*}
		\alpha_{m}-\alpha_{m-1}&=\frac{\alpha_{m-1}+\alpha_{m-1}^2h}{\alpha_{m-1}+k}-\alpha_{m-1}\\
		&= -\frac{\alpha_{m-1}(k-1)+\alpha_{m-1}^2(1-h)}{\alpha_{m-1}+k}\\
		&\leq  -\frac{\alpha_{*}(k-1)+\alpha_{*}^2(1-h)}{\alpha_{*}+k}<0,
\end{align*}
which implies $a_m\to -\yy $ as $m\to \yy$. This is a contradiction to the assumption $\alpha_m\geq (k-1)/h$. Hence, there is a finite integer $m_0\geq 2$ such that $\alpha_{m_0}< (k-1)/h$ and so $\beta_{m_0}<0$.

Using $\beta_{m_0}< 0$, we could show $h_\yy-g_\yy<\infty$,  $u(t,x)\to 0$ as $t\to\infty$ uniformly for $x\in [g(t), h(t)]$,  and $v(t,x)\to 1$ as $t\to\infty$ locally uniformly for $x\in\R$,
as in Case 1 above. The proof of Theorem \ref{th1.2} is now complete. \hfill $\Box$

\section{Proof of Theorems \ref{th1.3} and \ref{th1.4}}

		{\bf Proof of Theorem \ref{th1.3}.}  Let $V(t,x)$ denote the solution of \eqref{w} with $(d, J, a,b, w_0)=(d_2, J_2, \gamma, \gamma, v_0)$. The comparison principle yields $v(t,x)\leq V(t,x)$ for all $t\geq 0$ and $x\in \R$.

By \cite{DLZ-2021-JMPA}, with
\[
f_\rho(u):=u(1-\rho-u),
\]
for each small $\rho>0$,
there is a monotone function  $\phi=\phi_\rho\in  C^1((-\infty, 0])$ and a constant $c_1^\rho>0$ such that
	\begin{equation*}
		\begin{cases}
			\displaystyle d_1 \int_{-\infty}^0 J_1(x-y) \phi(y) dy - d_1 \phi(x)+ c^\rho_1\phi'(x) + f_\rho(\phi(x)) =0, &  -\infty < x< 0,\\
			\displaystyle \phi(-\infty) = 1,\ \ \phi(0) =0,\\
			\dd c_1^\rho=\mu\int_{-\infty}^{0}\int_{0}^{+\infty}J(x-y)\phi(x)dydx.
		\end{cases}
	\end{equation*}
	Moreover, $c_1^\rho\to c_1$ as $\rho\to 0$. Therefore we can fix $\rho>0$ small so that $c_1^\rho>C_2$.
	
	Fix small $\epsilon>0$ so that $C_2+\epsilon<c_1^\rho$. By Proposition E in Section 1 here  and Lemma 2.2 in \cite{DM}, there is $L_0>0$ such that 
		\begin{align*}
				\sup_{|x|\geq (C_2+\epsilon)t+L_0} v(t,x)\leq 	\sup_{|x|\geq (C_2+\epsilon)t+L_0} V(t,x)\leq \frac{\rho}{k}\ \mbox{ for all } \ t\geq 0.
		\end{align*}
	
Fix $S>0$ such that the support of $J_1$ is contained  in the interval $[-S,S]$, and fix  $K$ satisfying
 \label{key}
\begin{align}\label{K4.13}
	1>K>\frac{(1-2\epsilon)c_1^\rho+(C_2+\epsilon)}{2(1-2\epsilon)c_1^\rho}.
\end{align}
Then define, for some large constant $L>L_0/(2K-1)$, 
		\begin{align*}
			&\underline h(t):=c_1^\rho(1-2\epsilon) t+L, 
		\end{align*}
	and
	\begin{equation*}
	\underline u(t,x):=
	\begin{cases}
(1-\epsilon) \phi(x-\underline h(t)), & t\geq 0,\; K\underline h(t)\leq  |x|\leq \underline  h(t),\\
(1-\epsilon) \phi\big((2K-1)\underline h(t)-x\big), & t\geq 0,\; (2K-1)\underline h(t)\leq  |x|\leq K\underline h(t).
	\end{cases}
	\end{equation*}
		 The choice of $K$ implies
		\begin{align*}
			(2K-1)\underline h(t)>(C_2+\epsilon) t+L_0\ \mbox{ for } \ t\geq 0,
		\end{align*}
	and  so $kv(t,x)\leq \rho$ for $t\geq 0$ and $|x|\geq (2K-1)\underline h(t)$. Therefore, if $T>0$ and $h(t)>(2K-1)\underline h(t)$ for $t\in [0, T]$, then
	\begin{equation}\label{u-super}
			\begin{cases}
				\dd  u_t\geq  d_1 \int_{(2K-1)\underline h(t)}^{h(t)}  J_1(x-y)  u(t,y)\rd y-d_1 u+f_\rho( u), & t\in (0, T],\; x\in  ((2K-1)\underline h(t), h(t)),\\[4mm]
				\dd  h'(t)\geq   \mu\dd\int_{(2K-1)\underline h(t)}^{h(t)} \int_{h(t)}^{+\yy}  J_1(x-y) u(t,x)\rd y\rd x, & t\in (0, T],\\
				 u(t, (2K-1)\underline h(t))>0=u(t, h(t)), & t\in (0, T].
			\end{cases}
		\end{equation}
		
		We  will show that, with $\underline h$ denoting $\underline h(t)$, for sufficiently large $L$,
		\begin{equation}\label{4.12}
			\begin{cases}
				\dd \underline u_t\leq  d_1 \int_{(2K-1)\underline h}^{\underline h}  J_1(x-y) \underline u(t,y)\rd y-d_1\underline u+f_\rho(\underline u), & t>0,\; x\in  ((2K-1)\underline h,\underline h)\backslash \{K\underline h\},\\[4mm]
				\dd \underline h'\leq  \mu\dd\int_{(2K-1)\underline h}^{\underline h} \int_{\underline h}^{+\yy}  J_1(x-y)\underline u(t,x)\rd y\rd x, & t>0,\\
				\underline u(t,\underline h)=\underline u(t,(2K-1)\underline h) =0, & t\geq 0.
			\end{cases}
		\end{equation}
		
		Assuming the validity of \eqref{4.12} for the time being, let us see how $(u_0, h_0)$ can be chosen to obtain the desired conclusion.
		
Clearly, if $h_0\geq \underline h(0)$ and $u_0(x)\geq \underline u(0,x)$ for $x\in [(2K-1)\underline h(0),\underline h(0)]$,  then due to \eqref{u-super}, \eqref{4.12} and
$g(t)\leq -h_0<0<(2K-1)\underline h(t)$, 
\[
u(t,x)>0=\underline u(t, x) \mbox{ for } x=(2K-1)\underline h(t),\ t\geq 0,
\]
we can use the comparison principle   to conclude that 
\begin{align*}
 h(t)\geq \underline h(t),\ 
u(t,x)\geq \underline u(t,x) \mbox{ for }\ t\geq 0,\ x\in [(2k-1)\underline h(t),\underline h(t)]. 
\end{align*}
Hence, $h(t)\to \yy$ as $t\to\infty$.  

Similarly, we can show that if $h_0\geq \underline h(0)$ and $u_0(x)\geq \underline u(0, -x)$ for $x\in [-\underline h(0), -(2K-1)\underline h(0)]$,  then
\[
g(t)\leq -\underline h(t),\ 
u(t,x)\geq \underline u(t,-x) \mbox{ for }\ t\geq 0,\ x\in [-\underline h(t), -(2k-1)\underline h(t)]. 
\]
Therefore $g(t)\to-\infty$ as $t\to\infty$.

Hence the desired conclusion  $(g_\infty, h_\infty)=(-\infty, \infty)$ holds if the initial data $(u_0, h_0)$ satisfies
\[
h_0>\underline h(0), \ u_0(x)\geq \underline u(0,x) \mbox{ for } |x|\in [(2k-1)\underline h(0),\underline h(0)].
\]
		
To complete the proof, it remains to prove \eqref{4.12}.		
Clearly, the third equation in \eqref{4.12} follows directly from the definition of $\underline u$. 	Now, we verify the first two inequalities in \eqref{4.12}.  Recalling that $J_1$ has  compact support contained in $[-S, S]$, by a direct computation we obtain,  for large $L$ and all $t> 0$, 
		\begin{align*}
			 &\mu\dd\int_{(2K-1)\underline h}^{\underline h} \int_{\underline h}^{+\yy}  J_1(x-y)\underline u(t,x)\rd y\rd x
			 \geq  \mu\dd\int_{K\underline h}^{\underline h} \underline u(t,x)\int_{\underline h}^{+\yy}  J_1(x-y)\rd y\rd x\\
			 =&\mu(1-\epsilon)\dd\int_{(K-1)\underline h}^{0} \phi(x)\int_{0}^{+\yy}  J_1(x-y)\rd y\rd x=\mu(1-\epsilon)\dd\int_{-\yy}^{0} \phi(x)\int_{0}^{+\yy}  J_1(x-y)\rd y\rd x\\
			 =&(1-\epsilon )c_1>\underline h'(t).
		\end{align*}
	This showed the validity of the second inequality in \eqref{4.12}.

We next varify the first inequality in \eqref{4.12}.

{\bf Case 1}. $x\in (K\underline h,\underline h)$.

		In view of  the equation satisfied by $\phi$, 
		we deduce for $t>0$ and $x\in (K\underline h,\underline h]$,
		\begin{align*}
			\underline u_t(t,x)&=-(1-\epsilon)c_1^\rho(1-2\epsilon)\phi'(x-\underline h)\leq -(1-\epsilon)c_1^\rho\phi'(x-\underline h)\\ 
			&= (1-\epsilon) \lf[d_1  \int_{- \yy}^{\underline h}    J_1 (x-y)   \phi(y-\underline h)\rd y -d_1 \phi(x-\underline h)+ f_\rho(\phi(x-\underline h))\rr]\\ 
			&=d_1  (1-\epsilon)\int_{x-S}^{\underline h}    J_1 (x-y)  \phi(y-\underline h)\rd y -d_1  \underline u(t,x)+(1-\epsilon) f_\rho(\phi(x-\underline h)):=A.
		\end{align*}
		
	If $x\in [K\underline h+S,\underline h]$, then $x-y\geq S$ when $y\leq K\underline h$, and therefore
	\begin{align*}
	\underline u_t(t,x)\leq A&= d_1  \int_{K\underline h}^{\underline h}    J_1 (x-y)  \underline u(t,y)\rd y -d_1  \underline u(t,x)+(1-\epsilon) f_\rho(\phi(x-\underline h))\\
	&\leq d_1  \int_{K\underline h}^{\underline h}    J_1 (x-y)  \underline u(t,y)\rd y -d_1  \underline u+ f_\rho(\underline u),
\end{align*}
where we have used the fact that $(1-\epsilon) f_\rho(\phi(x-\underline h))\leq  f_\rho((1-\epsilon)\phi(x-\underline h))$.

For $x\in (K\underline h,K\underline h+S]$, we have
	\begin{align*}
	\underline u_t(t,x)\leq A&= d_1  \int_{K\underline h-S}^{\underline h}    J_1 (x-y)  \underline u(t,y)\rd y -d_1  \underline u(t,x)+f_\rho(\underline u)\\
	&\ \ \ \ +\int_{K\underline h-S}^{K\underline h}    J_1 (x-y)  [(1-\epsilon)\phi(y-\underline h)-\underline u(t,y)]\rd y+(1-\epsilon) f_\rho(\phi(x-\underline h))-f_\rho(\underline u)\\
	&= d_1  \int_{K\underline h-S}^{\underline h}    J_1 (x-y)  \underline u(t,y)\rd y -d_1  \underline u(t,x)+f_\rho(\underline u)\\
	&\ \ \ \ +\int_{K\underline h-S}^{K\underline h}    J_1 (x-y)  [(1-\epsilon)\phi(y-\underline h)-\underline u(t,y)]\rd y-(\epsilon-\epsilon^2) \phi^2(x-\underline h).
\end{align*}
From the fact that $\phi(s)\to 1-\rho$ as $s\to -\yy$, we deduce 
\begin{align*}
	&\int_{K\underline h-S}^{K\underline h}    J_1 (x-y)  [(1-\epsilon)\phi(y-\underline h)-\underline u(t,y)]\rd y-(\epsilon-\epsilon^2) \phi^2(x-\underline h)\\
	&=o(1)-(\epsilon-\epsilon^2)(1-\rho)< 0
\end{align*}
for large $L$ and all $t>0$. Therefore we also have
\[
\underline u_t(t,x)\leq d_1  \int_{K\underline h-S}^{\underline h}    J_1 (x-y)  \underline u(t,y)\rd y -d_1  \underline u+f_\rho(\underline u).
\]

{\bf Case 2}. $x\in [(2K-1)\underline h,K\underline h]$.

Note that for each $t\geq 0$, the function $\underline u(t,x)$ is symmetric with respect to $x=Kh(t)$:
\[
\underline u(t,x)=\underline u(t,2K\underline h-x).
\]
  Thus, for $x\in [(2K-1)\underline h,K\underline h]$,
\begin{align*}
	A_1:=&\ d_1  \int_{(2K-1)\underline h}^{\underline h}    J_1 (x-y)  \underline u(t,y)\rd y -d_1  \underline u(t,x)+f_\rho(\underline u(t,x))\\
	=&\ d_1  \int_{(2K-1)\underline h}^{\underline h}    J_1 (x-y)  \underline u(t,2K\underline h-y)\rd y -d_1  \underline u(t,2K\underline h-x)+f_\rho(\underline u(t,2K\underline h-x))\\
	=&\ d_1  \int_{(2K-1)\underline h}^{\underline h}    J_1 (2K\underline h-x-y)  \underline u(t,y)\rd y -d_1  \underline u(t,2K\underline h-x)+f_\rho(\underline u(t,2K\underline h-x)).
\end{align*}
Since $2K\underline h-x\in [K\underline h, \underline h]$, from the calculations in Case 1, one finds that $A_1\geq  0$. On the other hand, a direct computation gives, for $x\in ((2K-1)\underline h,K\underline h)$,
\begin{align*}
	u_t(t,x)&=2K(1-\epsilon)c_1^\rho(1-2\epsilon)\phi'\big(2Kh(t)-x\big)\leq 0.
\end{align*}
This immediately gives $\underline u_t(t,x)\leq 0\leq A_1$. Therefore, the first inequality in \eqref{4.12} always holds. 
\medskip

 The proof of  Theorem \ref{th1.3} is now complete.
\hfill $\Box$

\begin{remark}\label{rm4.1}
 From the proof of Theorem \ref{th1.3}, we actually know that
	\[\begin{cases}
	\limsup_{t\to\infty}\frac{h(t)}t\leq c_1,\\
	\liminf_{t\to\infty}\frac{h(t)}t\geq c_1^\rho,\end{cases}\  \begin{cases}
	\liminf_{t\to\infty}\frac{g(t)}t\geq -c_1,\\
	\limsup_{t\to\infty}\frac{g(t)}t\leq -c_1^\rho,\end{cases}
	\]
	\[\begin{cases}
	\liminf_{t\to\infty} \min_{(C_2+\epsilon)t\leq |x|\leq (c^\rho_1-3\epsilon)t}u(t,x)\geq 1-\epsilon,\\
	\lim_{t\to\infty} \max_{|x|\geq (C_2+\epsilon)t}v(t,x)=0.
	\end{cases}
	\]

	 \end{remark}
 
 \bigskip
 
 {\bf Proof of Theorem \ref{th1.4}.}	In this proof, for convenience we will replace the initial population range $[-h_0, h_0]$ of $u$ by some general interval $[g_0, h_0]$ not necessarily symmetric about $x=0$.
Note that such a non-symmetric case can always be reduced to a symmetric one by a simple shift of the variable $x$: if $(u(t,x), v(t,x),  h(t),  g (t))$ has the required property with initial range $(g_0, h_0)$ for $u$, then, with $x_0:= \frac{g_0+h_0}2$,
\[
(\tilde u(t,x), \tilde v(t,x), \tilde h(t), \tilde g (t)):=(u(t, x-x_0), v(t, x-x_0), h(t)-x_0, g (t)-x_0)
\]
 is a solution with symmetric initial range $(-\frac{h_0-g_0}2, \frac{h_0-g_0}2)$ for $u$ and the desired property. So no generality is lost by considering a general initial range
$[g_0, h_0]$ of $u$.

Since $v_0$ is compactly supported, from the proof  of Theorem \ref{th1.3} we see that a lower solution can  be constructed on one side of the population range  $[g(t),h(t)]$ of $u$, for example  in a subset of
$[0, h(t)]$ to guarantee that $h_\infty=\infty$,  while there are no specific restrictions imposed on the other side of the population range of $u$, namely the side $[g(t), 0]$.  This suggests  that it is perhaps possible to choose an initial function $u_0(x)$ which is sufficiently small for $x\leq 0$ such that  $g(t)$ remains uniformly bounded for all $t$ while $h(t)\to \yy$ as $t\to \yy$. We show below that this is indeed possible.

{\bf Step 1.} Construction of an upper solution $(\bar u, \bar g)$.
 
Let $\alpha>1$ be a fixed number, and $T>0$ be a large constant to be specified. Denote $\td t=t+T$ for $t\in \R$, and define 
		\begin{align*}
	&\bar g=\bar g(t):=-\tilde L+\delta (\ln \td t)^{1-\alpha} ,  \ \ \ t\geq 0
\end{align*}
and
\begin{equation*}
	\bar u(t,x):=
	\begin{cases}
	\dd\frac{M}{\td t[(\ln \td t)^{\alpha}-2\bar g]}(x-2\bar g), & t\geq 0,\; x\in [\bar g,(\ln \td t)^{\alpha}],\\
\dd 	\frac{1-M/\td t}{C_2\td t/2-(\ln\td t)^{\alpha}}[x-(\ln \td t)^{\alpha}]+\frac{M}{\td t}, & t\geq 0,\; x\in ((\ln \td t)^{\alpha},+\yy),
	\end{cases}
\end{equation*}
for some $\delta>0$, $M>0$  and large $\tilde L>0$. Let us recall that $C_2$ is the spreading speed of $v$ in the absence of $u$. 
Clearly, $\bar u$ is continuous and piecewise linear in $x$.

We next  show that, with $\bar g$ denoting $\bar g(t)$,  
	\begin{equation}\label{4.15}
	\begin{cases}
		\dd \bar u_t\geq   d_1\! \int_{\bar g}^{+\yy} \!\!\! J_1(x\!-\!y) \bar u(t,y)\rd y\!-\!d_1\bar u-\frac{k(1-h)-1}{2}\bar u, & t>0,\; |x|\in  (\bar g,C_2 \td t/2)\backslash \{(\ln \td t)^{\alpha}\},\\[4mm]
		\dd \bar g'\leq  -\mu\dd\int_{\bar g}^{+\yy} \bar u(t,x)\int_{-\yy}^{\bar g}  J_1(x-y)\rd y\rd x, & t>0,\\
		\bar u(t,\bar g)\geq 0, \
			\bar u(t,x)\geq 1, & t\geq 0,\; x\geq  C_2\td t/2.
	\end{cases}
\end{equation}
It is clearly that the two inequalities in the third line of \eqref{4.15} follow directly from  the definition of $\bar u$. In the following we check the other inequalities in \eqref{4.15}.

{\bf (a)} We verify the second inequality in \eqref{4.15}. 

Suppose that the support of  $J_1$ is contained  in the interval $[-S,S]$ for some $S>0$. Then for $T\gg S$, 
\begin{align*}
	&-\mu\dd\int_{\bar g}^{+\yy} \bar u(t,x)\int_{-\yy}^{\bar g}  J_1(x-y)\rd y\rd x=-\mu\dd\int_{\bar g}^{\bar g+S} \bar u(t,x)\int_{-\yy}^{\bar g}  J_1(x-y)\rd y\rd x\\
	\geq& -\mu\dd\int_{\bar g}^{\bar g+S} \bar u(t,x)\rd x\geq -\mu S \bar u(t,\bar g+S)=	\dd\frac{-\mu M}{\td t[(\ln \td t)^{\alpha}-2\bar g]}(-\bar g+S)\\
	\geq &\	\dd\frac{-\mu M(\tilde L+S)}{\td t[(\ln \td t)^{\alpha}+2\tilde L]}\geq \frac{-2\mu M(\tilde L+S)}{\td t(\ln \td t)^{\alpha}}.
\end{align*}
Note that 
\begin{align*}
	\bar g'(t)= -\frac{(\alpha-1)\delta}{\td t(\ln \td t)^{\alpha}}\ \mbox{ for }\ t\geq 0.
\end{align*}
Hence, the desired  inequality holds if 
\begin{align*}
	\delta\geq \frac{2\mu M(\tilde L+S)}{\alpha-1}.
\end{align*} 

   {\bf (b)} We prove  the first inequality in \eqref{4.15}. 
   
   A direct calculation gives, for $x\in [\bar g,(\ln \td t)^{\alpha})$ and $t>0$,
   \begin{align*}
   	\bar u_t=&\ 	\frac{M}{\td t[(\ln \td t)^{\alpha}-2\bar g]}(-2\bar g')-	\frac{M}{\td t[(\ln \td t)^{\alpha}-2\bar g]}(x-2\bar g)\frac{(\ln \td t)^{\alpha}+\alpha (\ln \td t)^{\alpha-1}-2\bar g-2\td t\bar g'}{\td t[(\ln \td t)^{\alpha}-2\bar g]}\\
   	\geq &-\bar u \frac{(\ln \td t)^{\alpha}+\alpha (\ln \td t)^{\alpha-1}-2\bar g-2\td t\bar g'}{\td t[(\ln \td t)^{\alpha}-2\bar g]}\geq -\frac{k(1-h)-1}{4}\bar u\ \ \ {\rm for\ large }\ T>0,
   \end{align*}
where we have used the fact that $0>\bar g'=\dd -\frac{(\alpha-1)\delta}{\tilde t (\ln \tilde t)^\alpha}$. 

For $x>(\ln \tilde t)^\alpha$ and $t>0$, we have
{\small \begin{align*}
		\bar u_t=&\ 	\frac{1-M/\td t}{C_2\td t/2-(\ln\td t)^{\alpha}}[-\alpha \td t^{-1}(\ln \td t)^{\alpha-1}]-\frac{M}{\td t^2}\\
		&+	 \frac{x-(\ln \td t)^{\alpha} }{C_2\td t/2-(\ln\td t)^{\alpha}}	\frac{M\td t^{-2}(C_2\td t/2-(\ln\td t)^{\alpha})-(1-M/\td t)[C_2/2-\alpha\td t^{-1}(\ln\td t)^{\alpha-1}]}{C_2\td t/2-(\ln\td t)^{\alpha}}\\
\geq & \	\frac{4}{C_2\td t}[-\alpha \td t^{-1}(\ln \td t)^{\alpha-1}]-\frac{M}{\td t^2}-\dd 	\frac{x-(\ln \td t)^{\alpha}}{C_2\td t/2-(\ln\td t)^{\alpha}}\dd 	\frac{2}{C_2\td t}\\
\geq& -\frac{k(1-h)-1}{4}\bar u\ \ \ {\rm for\ large}\ T>0.
\end{align*}}

We now estimate the nonlocal diffusion term for such $x$ and $t$, and show that
\begin{align}\label{4.16}
		A:=d_1 \int_{\td g}^{+\yy}  J_1(x-y) \bar u(t,y)\rd y-d_1\bar u(t,x)\leq \frac{k(h-1)-1}{4}\bar u \mbox{ for all large } T>0.
\end{align} 

Since $\bar u(t,x)$ is a linear function of $x$ for  $x\in [\bar g,(\ln \td t)^{\alpha})$ and for $x\in ((\ln \td t)^{\alpha},+\yy)$, respectively, we have, for $x\in [\bar g,(\ln \td t)^{\alpha}-S]\cup [(\ln \td t)^{\alpha}+S,C_2 \td t/2)$ and $t>0$,
\begin{align*}
	A&=d_1 \int_{\bar g}^{+\yy}  J_1(x-y) \bar u(t,y)\rd y-d_1\bar u(t,x) \leq 	d_1 \int_{x-S}^{x+S}  J_1(x-y) \bar u(t,y)\rd y-d_1\bar u(t,x)\\
=&\ d_1 \int_{x}^{x+S}  J_1(x-y)[ \bar u(t,y)+\bar u(t,2x-y)]\rd y-d_1\bar u(t,x)\\
=&\ 2d_1 \bar u(t,x)\int_{x}^{x+S}  J_1(x-y)\rd y-d_1\bar u(t,x)=0.
\end{align*}
Hence \eqref{4.16} holds, 

For $x\in [(\ln \td t)^{\alpha}-S,(\ln \td t)^{\alpha})$ and $t>0$, we introduce the linear function (in $x$)
\begin{align*}
	 \td u_1:=	\dd\frac{M}{\td t[(\ln \td t)^{\alpha}-2\bar g]}(x-2\bar g), 
\end{align*}
and obtain, for large $T>0$,
\begin{align*}
	A=&\ d_1 \int_{\bar g}^{+\yy}  J_1(x-y) \bar u(t,y)\rd y-d_1\bar u(t,x) =d_1 \int_{x-S}^{x+S}  J_1(x-y) \bar u(t,y)\rd y-d_1\bar u(t,x)\\
	=&\ d_1 \int_{x-S}^{x+S}  J_1(x-y) \td  u_1(t,y)\rd y-d_1\bar u(t,x)+d_1 \int_{(\ln \td t)^{\alpha}}^{x+S}  J_1(x-y) [\bar  u(t,y)-\td u_1(t,y)]\rd y\\
	=&\ d_1 \int_{(\ln \td t)^{\alpha}}^{x+S}  J_1(x-y) [\bar  u(t,y)-\td u_1(t,y)]\rd y\leq d_1 \int_{(\ln \td t)^{\alpha}}^{(\ln \td t)^{\alpha}+S}  J_1(x-y) [\bar  u(t,y)-M/\td t]\rd y\\
	= &\ \dd 	\frac{1-M/\td t}{C_2\td t/2-(\ln\td t)^{\alpha}}d_1 \int_{(\ln \td t)^{\alpha}}^{(\ln \td t)^{\alpha}+S}  J_1(x-y)[y-(\ln \td t)^{\alpha}] \rd y\\
	\leq&\   \dd 	\frac{d_1 S(1-M/\td t)}{C_2\td t/2-(\ln\td t)^{\alpha}}\leq  \dd 	\frac{4d_1 S}{C_2\td t}.
\end{align*}
On the other hand, by the definition of $\bar u$,  for large $T>0$, we have
\begin{align*}
	&\frac{k(1-h)-1}{4}\bar u(t,x)=\frac{k(1-h)-1}{4}	\dd\frac{M}{\td t[(\ln \td t)^{\alpha}-2\bar g]}(x-2\bar g)\\
	\geq&\ \frac{k(1-h)-1}{4}	\dd\frac{M}{\td t[(\ln \td t)^{\alpha}-2\bar g]}((\ln \td t)^{\alpha}-S-2\bar g)\geq \frac{k(1-h)-1}{8\td t}M.	
\end{align*}
Thus it is clear that \eqref{4.16} holds if
\begin{align}\label{4.17}
	M\geq \frac{32d_1 S}{C_2[k(1-h)-1]}.
\end{align}

For $x\in ((\ln \td t)^{\alpha},(\ln \td t)^{\alpha}+S)$ and $t>0$, we similarly derive, for large $T>0$ and linear function (in $x$)
\begin{align*}
	 \td u_2:=	\frac{1-M/\td t}{C_2\td t/2-(\ln\td t)^{\alpha}}[x-(\ln \td t)^{\alpha}]+\frac{M}{\td t}, 
\end{align*}
\begin{align*}
	A=&\ d_1 \int_{x-S}^{x+S}  J_1(x-y) \bar u(t,y)\rd y-d_1\bar u(t,x)\\
	=&\ d_1 \int_{x-S}^{x+S}  J_1(x-y) \td  u_1(t,y)\rd y-d_1\bar u(t,x)+d_1 \int^{(\ln \td t)^{\alpha}}_{x-S}  J_1(x-y) [\bar  u(t,y)-\td u_2(t,y)]\rd y\\
	=&\ d_1 \int^{(\ln \td t)^{\alpha}}_{x-S}  J_1(x-y) [\bar u(t,y)-\td u_2(t,y)]\rd y\leq d_1 \int^{(\ln \td t)^{\alpha}}_{(\ln \td t)^{\alpha}-S} J_1(x-y) [M/\td t-\td u_2(t,y)]\rd y\\
	= &\  \dd 	\frac{1-M/\td t}{C_2\td t/2-(\ln\td t)^{\alpha}}d_1 \int_{(\ln \td t)^{\alpha}}^{(\ln \td t)^{\alpha}+S}  J_1(x-y)[y-(\ln \td t)^{\alpha}] \rd y\leq  \dd 	\frac{d_1 S(1-M/\td t)}{C_2\td t/2-(\ln\td t)^{\alpha}}\leq  \dd 	\frac{4d_1 S}{C_2\td t}.
\end{align*}
Since $\bar u(t,x)\geq M/\td t$ for $x\geq (\ln \td t)^{\alpha}$, we see immediately that  \eqref{4.16} is satisfied provided \eqref{4.17} holds. 

By \eqref{4.16} and our  estimates of $u_t$ we see immediately that \eqref{4.15} holds for large $T>0$ and $M$ satisfying \eqref{4.17}.
\medskip

{\bf Step 2}. We choose $(u_0, h_0, g_0)$ to have the desired long-time limit for $g(t)$ and $h(t)$. 

Firstly we fix $(T, \tilde L, \delta, M)$ such that $\bar g(t)$ and $\bar u(t,x)$ satisfy \eqref{4.15}. So in particular,
\[
\bar u(T_1, x)\geq 1 \mbox{ for } x\geq \frac{C_2}2 (T+t).
\]
Without loss of generality we may assume that $T$ has been chosen large enough such that 
\[
\frac{C_2}2T>S+1.
\]

We aim to choose $(u_0, g_0, h_0) $ such that $h_\infty=\infty$ and
\begin{equation}\label{T-condition}\begin{cases}
g(T_1)\geq \bar g(T_1),\ u(T_1, x)\leq \bar u(T_1, x) \mbox{ for } x\in [g(T_1), \min\{h(T_1), \frac{C_2}2(T+T_1)],\\
u(1-u-kv)\leq -\frac{k(1-h)-1}2 u \mbox{ for }  t>T_1, \ x\in [\bar g(t), \min\{h(t), \frac{C_2}2(T+t)\}].
\end{cases}
\end{equation}
If these inequalities are proven, then we can use the comparison principle to conclude that
\[
g(t)\geq \bar g(t),\ u(t, x)\leq \bar u(t, x) \mbox{ for } t>T_1, \ x\in [g(t), \min\{h(t), \frac{C_2}2(T+t)\}],
\]
which clearly implies $g_\infty>-\infty$. 

To complete the proof, it remains to choose $(u_0, g_0, h_0)$ such that $h_\infty=\infty$ and all the inequalities in \eqref{T-condition} hold.

Since $\limsup_{t\to\infty} u(t,x)\leq 1$ uniformly in $x$,  for any given small $\bar\epsilon>0$, we can find
 $T_0>0$ large such that 
 \[
 v_t\geq d_2\int_\R J_2(x-y)v(t,y)dy-d_2 v+\gamma v(1-v-h(1+\bar\epsilon)) \mbox{ for } t\geq T_0,\ x\in\R.
 \]
 It follows that $v(T_0+t,x)\geq V(t,x)$ where $V$ is the unique solution of \eqref{w} with $(d, J, a, b, w_0)=(d_2, J_2, \gamma[1-h(1+\epsilon)], \gamma, v(T_0, x))$.
 By Proposition E in Section 1, for any small $\tilde\epsilon>0$,
 \[
 \lim_{t\to\infty}\max_{|x|\leq (C_2-\tilde\epsilon)t}|V(t,x)-1-h(1+\epsilon)|=0.
 \]
 Therefore we can find $T_1>T_0$ such that
\begin{align*}
v(t,x)>1-h-\bar\epsilon\ \mbox{ for } \  t\geq T_1,\  |x|\leq \frac{3C_2 }{4} (T+t).
\end{align*}
This clearly implies the validity of the inequality in the second line of \eqref{T-condition}.

Since $\bar g(T_1)=-\tilde L+\frac\delta{(\ln T_1)^{\alpha-1}}$, we may assume that $T_1$ is large enough to also guarantee
\[
\bar g(T_1)<0.
\]

 Since $\bar u(T, x)>0$ for $x\in [\bar g(T_1), \min\{h(T_1), \frac{C_2}2(T+T_1)\}]$, there exists $\epsilon_1>0$ such that
\[
\mbox{$\bar u(T, x)\geq \epsilon_1$ for $x\in [\bar g(T_1), \min\{h(T_1), \frac{C_2}2(T+T_1)\}]$.}
\]
By Proposition E, equation \eqref{w} with $(d, J, a, b)=(d_1, J_1, 1,1)$ has a traveling wave $\phi_1(x)$ with speed $C_1$. Let $\tilde L_0>0$ be chosen such that 
\[
\mbox{$\phi_1(x)<\tilde\epsilon_1:=\min\{\epsilon_1, \frac 1{\mu ST_1}\}$ for $x\geq \tilde L_0$. }
\]
We then fix $\hat L_0>0$ such that
\[
\bar g(T_1)-C_1T_1+\hat L_0>\tilde L_0.
\]

Let
$\underline h(t)$ and $\underline u(t,x)$ be defined as in the proof of Theorem \ref{th1.3} so that \eqref{4.12} holds, and note that \eqref{4.12} holds for every large $L$
used in the definition of $\underline h(t)$. We recall that
\[
\underline h(0)=L, \ \underline u(0,x)=\begin{cases}(1-\epsilon)\phi(x-L) &\mbox{ for } x\in [KL, L],\\
(1-\epsilon)\phi((2K-1)L-x)& \mbox{ for } x\in [(2K-1)L, KL].
\end{cases}
\]
Moreover, $(2K-1)L>\frac{C_2}{c_1^\rho}L$, where $c_1^\rho$ is defined in the proof of Theorem \ref{th1.3}.

Let $L_1>0$ be chosen such that $\phi_1(x)\geq 1-\epsilon$ for $x\leq -L_1$. We now assume that $L$ in the definition of $\underline h(t)$ is chosen so large that
apart from \eqref{4.12} we also have
\[\begin{cases}
L>\frac{C_2}2(T+T_1),\\
(2K-1)L>\bar g(T_1)+\epsilon_1\mu ST_1+S,\\
 (2K-1)L\geq L_1+L_0+C_1T_1+  \frac{C_2}2(T+T_1).
 \end{cases}
\]

Define
\[
U(t,x):=\phi_1(-x-C_1t-\hat L_1),\ \hat L_1:=L_1-(2K-1)L.
\]
Clearly for $x\in [(2K-1)L, L]$, we have
\[
U(0,x)\geq \phi_1(-(2K-1)L-\hat L_1)=\phi_1(-L_1)\geq 1-\epsilon> \underline u(0, x).
\]

We now fix $h_0\geq L$, and note that the above estimate for $U(0, x)$ and $(2K-1)L>\bar g(T_1)+\epsilon_1\mu S T_1+S$ allow us to choose $(u_0, g_0)$ such that
\begin{equation}\label{u_0}
\begin{cases}
g_0:=\bar g(T_1)+\tilde \epsilon_1\mu ST_1,\\
u_0(x)\leq U(0, x)& \mbox{ for } x\in [g_0, h_0],\\
u_0(x)\geq \underline u(0, x) &\mbox{ for } x\in [(2K-1)L, L].
\end{cases}
\end{equation}

By the proof of Theorem \ref{th1.3}, the third inequality in \eqref{u_0} guarantees that $h_\infty=\infty$. Moreover, we note that
\[
h(t)\geq h_0\geq L>\frac{C_2}2(T+t) \mbox{ for } t\in [0, T_1].
\]

Since $U(t,x)$ satisfies
\[
U_t=d_1\int_{\R} J_1(x-y)U(t,y)dy-d_1U+U(1-U) \mbox{ for } t>0,\ x\in\R,
\]
while
\[
u_t\leq d_1\int_{\R} J_1(x-y)u(t,y)dy-d_1u+u(1-u) \mbox{ for } t>0,\ x\in (g(t), h(t)),
\]
by the comparison principle and the first inequality in \eqref{u_0} we have
\[
u(t,x)\leq U(t, x) \mbox{ for } t>0, \ x\in (g(t), h(t)).
\]
Hence for $t\in [0, T_1]$ and $x\in [g(t),  \frac{C_2}2(T+t)]$, we have
\[\begin{aligned}
u(t, x)\leq U(t, x)&=\phi_1(-x-C_1t-\hat L)\\
&\leq \phi_1(-\frac{C_2}2(T+t)-C_1t-\hat L)\\
&=\phi_1(-\frac{C_2}2(T+T_1)-C_1T_1-L_1+(2K-1)L)\\
&\leq \phi_1(L_0)\leq \tilde\epsilon_1.
\end{aligned}
\]
In particular, for $x\in [g(T_1),  \frac{C_2}2(T+T_1)]$, we have
\[\begin{aligned}
u(T_1, x)\leq  \tilde \epsilon_1\leq \bar u(T_1, x).
\end{aligned}
\]

We now estimate $g(T_1)$. For $t\in (0, T_1]$, 
\[
g(t)+S\leq g_0+S=\bar g(T_1)+\tilde \epsilon_1\mu ST_1+S<1+S<\frac{C_2}2(T+t),
\]
and so $u(t,x)\leq \tilde\epsilon_1$ for $x\in [g(t), g(t)+S]$. It follows that, for $t\in (0, T_1]$,
\[\begin{aligned}
g'(t)&=-\mu\int_{-\infty}^{g(t)}\int_{g(t)}^{h(t)}J_1(x-y)u(t,y)dydx\\
&=-\mu\int_{-\infty}^{g(t)}\int_{g(t)}^{g(t)+S}J_1(x-y)u(t,y)dydx\\
&\geq -\mu S\tilde \epsilon_1.
\end{aligned}
\]
Therefore
\[
g(T_1)\geq g_0-\mu S\tilde \epsilon_1T_1=\bar g(T_1).
\]

Now all the inequalities in \eqref{T-condition} are satisfied and
the proof of the theorem is complete.  
	\hfill $\Box$

\begin{remark}\label{rm:4.2}
From the above proof of Theorem \ref{th1.4} we can easily obtain the following estimates:
	\[\begin{cases}
	\dd\limsup_{t\to\infty}\frac{h(t)}t\leq c_1,\\
	\dd\liminf_{t\to\infty}\frac{h(t)}t\geq c_1^\rho,
	\end{cases} \ \ 
	\begin{cases}
	\dd\liminf_{t\to\infty} \min_{(C_2+\epsilon )t\leq x\leq (c^\rho_1-3\epsilon)t}u(t,x)\geq 1-\epsilon,\\
	\dd\lim_{t\to\infty} \max_{|x|\geq (C_2+\epsilon)t}v(t,x)=0,
	\end{cases}
	\]
	 and for any  given positive function $\xi(t)=o(t)$ as $t\to\infty$, 
	\[\begin{cases}
	\dd\lim_{t\to\infty} \max_{x\in [g(t), \xi(t)]}u(t,x)=0,\\
	\dd\lim_{t\to\infty} \min_{-(C_2-\epsilon)t\leq x\leq \xi(t)}v(t,x)=1.
	\end{cases}
	\]
\end{remark}

\section{General criteria for spreading and vanishing}

In this section, we prove Theorem \ref{th1.1}, which gives sharp conditions for spreading or vanishing to happen. Theorem \ref{v-1} will follow from one of the lemmas used to prove Theorem \ref{th1.1}.

The comparison principle infers that if vanishing happens for a particular value $\mu_0$ of $\mu$, then it happens also for any $\mu<\mu_0$. Similarly, if spreading happens for $\mu=\mu_1$, then it also happens for $\mu>\mu_1$.

To emphasize the dependence of the solution $(u,v, g,h)$ on the parameter $\mu$, we denote $(u,v,g,h)=(u^\mu, v^\mu, g^\mu,h^\mu)$, and similarly $g_\infty=g_\infty^\mu$ and $h_\infty=h_\infty^\mu$.  Define
\begin{align}\label{mu1}
 \Omega^{\mu}:=\{\mu>0: h_\infty^\mu-g_\infty^\mu<\infty\},\ \mu_*:=\begin{cases}\sup \Omega^{\mu} &\mbox{ if } \Omega^\mu\not=\emptyset,\\ 0 &\mbox{ if } \Omega^\mu=\emptyset.\end{cases}
\end{align}
\begin{lemma}\label{lem:mu_*}
 The following conclusions hold true:
 \begin{itemize}
 	\item[{\rm (i)}] $\mu_*=0$ implies $h^\mu_\infty-g^\mu_\infty=\yy$  for every $\mu>0$;
 	\item[{\rm (ii)}] $\mu_*\in (0, \infty)$ implies $h^\mu_\infty-g^\mu_\infty<\yy$  for every $\mu\in (0,\mu_*]$, and $h_\infty-g_\infty=\yy$  for every $\mu> \mu_*$;
 	\item[{\rm (iii)}] $\mu_*=\yy$ implies $h^\mu_\infty-g^\mu_\infty<\yy$  for every $\mu>0$.
 \end{itemize}
\end{lemma}
\begin{proof}
These conclusions follow directly from the definition of $\mu_*$, except that in  part (ii),  the conclusion $h_\infty-g_\infty<\yy$ for $\mu= \mu_*$ is derived   by a similar discussion as in the proof of \cite[Lemma 3.14]{Cao-2019-JFA}.
\end{proof}

 Note that part (i) of Theorem \ref{th1.1} follows directly from Theorem A, while the first half of part (iv) is a consequence of Lemma \ref{lem:mu_*}. To complete the proof
 of Theorem \ref{th1.1}, it remains to consider the following two cases:
\begin{align*}
	d_1 > 1 - k>0\ \  \mbox{ and }\ \  k>1>h.
\end{align*}

\begin{lemma}\label{l4.3}
Suppose $d_1 > 1 - k>0$.
\begin{itemize}
	\item[{\rm (i)}]  If $2h_0 \geq l_{1-k}$,  then $\mu_*=0$.
		\item[{\rm (ii)}]   If $2h_0 < l_{1-k}$, then  $\mu_*\in [0,\yy)$. 			
		\item[{\rm (iii)}]   If $d_1>1$ and $2h_0<l_1$, then $\mu_*\in (0, \infty)$.
			
\end{itemize}
\end{lemma}
\begin{proof}
{\rm (i)} Arguing indirectly, we suppose that for some $\mu>0$, vanishing happens, namely $h_\infty - g_\infty <\yy$. Then by \eqref{3.3} we obtain $\lambda_p(\mathcal{L}_{(g_\infty, h_\infty)}) \leq k-1$. 
Since $h_\infty-g_\infty>2h_0\geq l_{1-k}$, by Proposition C  we deduce  $\lambda_p(\mathcal{L}_{(g_\infty, h_\infty)}) > k-1$. This  contradiction shows that $h_\infty - g_\infty =\yy$
always holds. 
\medskip

{\rm (ii)} It suffices to prove that  $h_\infty - g_\infty = \infty$  for all large $\mu$.
In view of the conclusion (i) above, it is sufficient to find some $t_0 > 0$ such that
\begin{align}\label{4.1}
	h(t_0)-g(t_0)\geq l_{1-k},
\end{align}
which  is a consequence of \cite[Lemma 3.9]{DY-2022-DCDS}. Indeed, from the equation for $u$ we obtain, for  some positive  constant $C>1+\max_{x\in [-h_0,h_0]}u_0(x)+\max_{x\in \R}kv_0(x)$,
\begin{align}
	\begin{cases}
	\dd	u_{t}(t, x) \geq d_1\int_{g(t)}^{h(t)}J_1(x - y)u(t, y)\rd y - d_1u - Cu,   & t > 0, ~ g(t) < x < h(t), \\
		u(t, g(t)) = u(t, h(t)) = 0, &t\geq 0,\\
\dd		h^\prime(t) \geq \mu\int_{g(t)}^{h(t)}\int_{h(t)}^{\infty}J_1(x - y)u(t, x)\rd y\rd x, &t\geq 0,\\
	\dd	g^\prime(t) \leq - \mu\int_{g(t)}^{h(t)}\int_{-\infty}^{g(t)}J_1(x - y)u(t, x)\rd y\rd x, &t\geq 0,\\
		u(0, x) = u_{0}(x), &
		\vert x \vert\leq h_0,\\
		h(0) = -g(0) = h_0. &
		\nonumber
	\end{cases}
\end{align}
Hence we can use \cite[Lemma 3.9]{DY-2022-DCDS}  to conclude that for any given $t_0>0$, there exists $\mu_1>0$ such that  \eqref{4.1} holds for all $\mu >\mu_1$.  

\medskip

(iii) Suppose  $d_1>1$ and $2h_0<l_1$. Since $l_1<l_{1-k}$ by (ii) we have $\mu_*<\infty$.  It remains to show that $h_\infty - g_\infty < \infty$  for sufficient small $\mu$. 
We will demonstrate this by constructing an upper solution following the method of  \cite[Theorem 3.12]{Cao-2019-JFA} and \cite[Lemma 4.7]{DuNi2020N}. Since $2h_0 < l_1$, we have  $\lambda_p^\epsilon:=\lambda_p(\mathcal{L}_{(-h_0-\epsilon, h_0+\epsilon)} )<-1$ for small $\epsilon>0$ satisfying $2h_0+2\epsilon<l_1$. Let $\phi=\phi_\epsilon$ be a positive eigenfunction  corresponding to $\lambda_p^\epsilon$.  
Define, for $t\geq 0$, $x\in [-h_0-\epsilon, h_1+\epsilon]$,
 \begin{align*}
 	&\bar h(t):=h_0+\epsilon (1-e^{-\delta t}),\ \  \bar g(t):=-\ol h(t) \ \ {\rm and}\ \ \bar u(t,x):=Me^{-\delta t} \phi(x)
 \end{align*}
with
\begin{align*}
	 \delta:=-\lambda_p^\epsilon/2>0, \; M:= \delta \frac{\epsilon}{\mu}\Big(\!\int_{-h_0-\epsilon}^{h_0+\epsilon}\phi(x)\rd x\Big)^{-1}>0.
\end{align*}
 Clearly $\ol h(t)\in [h_0, h_0+\epsilon)\subset [h_0, l_*)$ for $t\geq 0$.   We next show  that for all small $\mu>0$,  $(\bar u,\bar g,\bar h)$ is an upper solution to the problem satisfied by  $(u,g,h)$, when $v(t,x)$ is viewed as a known function. If this is proved, then it follows from the comparison principle, see   Lemma \ref{lemmacomp1}, that 
 \[
 [g(t), h(t)]\subset[\bar g(t), \bar h(t)] \mbox{ and hence } h_\infty-g_\infty\leq \bar h(\infty)-\bar g(\infty)=2h_0+2\epsilon,
 \]
 as desired. 

 It remains to prove that $(\bar u,\bar g,\bar h)$ is an upper solution.  A simple computation gives
 \begin{align*}
 	&\dd \bar u_t-d_1 \int_{\bar g(t)}^{\bar h(t)}J_1(x-y)\bar u(t,y)\rd y+d_1\bar u(t,x)-\bar u(1-\bar u-kv)\\
 	\geq\ & \dd \bar u_t-d_1 \int_{-h_0-\epsilon}^{h_0+\epsilon}J_1(x-y)\bar u(t,y)\rd y+d_1\bar u(t,x)-\bar u\\
 	=\ &-\delta \bar u-\lambda_p^\epsilon\bar u=\delta\bar u\geq 0\ \ \ \ \ \ \ {\rm for}\  t>0, \; x\in  (\bar g(t),\bar h(t)).
 \end{align*}
 Recalling that   $[\bar g(t),\bar h(t)]\subset (-h_0-\epsilon,h_0+\epsilon)$, we further deduce
 \begin{align*}
 	\mu \int_{\bar g(t)}^{\bar h(t)}\int_{\bar h(t)}^{\yy}J_1(x-y)\bar u(t,x)\rd y\rd x\leq\ & 
 	\mu \int_{\bar g(t)}^{\bar h(t)}\bar u(t,x)\rd x=\mu Me^{-\delta t} \int_{\bar g(t)}^{\bar h(t)}\phi(x)\rd x\\
 	\leq\ & \mu Me^{-\delta t} \int_{-h_0-\epsilon}^{h_0+\epsilon}\phi(x)\rd x= \epsilon \delta e^{-\delta t}=\bar h'(t) \mbox{ for } t> 0,
 \end{align*}
 and by symmetry,
 \begin{align*}
 	-\mu \int_{\bar g(t)}^{\bar h(t)}\int_{-\yy}^{\bar g(t)}J_1(x-y)\bar u(t,x)\rd y\rd x\geq \bar g'(t)\ \mbox{ for } \ t> 0. 
 \end{align*}
 It is clear that $\bar u(t,\bar g(t))>0$ and $\bar u(t,\bar h(t))>0$ for all $t\geq 0$, and by appropriately selecting a small value for $\mu>0$, we can make $M$ as large 
 as we want and hence ensure that $u_0(x) \leq \bar{u}(0, x)$ for $x \in [-h_0, h_0]$. 
 The above calculations indicate that  $(\bar u,\bar g,\bar h)$ is an upper solution for $(u,g,h)$, as we wanted. 
 \end{proof}
 
 \begin{lemma}\label{lem:u-inf}
	If   $k> 1>h$, then  for any $(u_0,v_0)$ satisfying \eqref{KnK1.3}, and any $L>0$, we have
\begin{equation}\label{0-1} \begin{cases}
\lim_{t\to\infty} u(t,x)=0 &\mbox{ uniformly for } x\in [-L, L],\\
\lim_{t\to\infty} v(t,x)=1 &\mbox{ uniformly for } x\in [-L, L].
\end{cases}
\end{equation}
Moreover, $h_\infty-g_\infty<\infty$ for all small $\mu>0$, and hence
	$\mu_*>0$.
\end{lemma}
\begin{proof}
	Let $\bar u,\underline v$ be the solution of 
	\begin{align*}
		\begin{cases}
			\dd	\bar u_t =  d_1\int_{\R}J_1(x - y)\bar u(t, y)\rd y - d_1\bar u + \bar u(1 - \bar u - k \underline v),   &
			t > 0,\ x\in \R,\\[4mm]
				\dd	\underline v_t =  d_2\int_{\R}J_2(x - y)\underline v(t, y)\rd y - d_2 \underline v + \gamma \underline v(1 - \underline v- h \bar u),   &
		t > 0,\ x\in \R
		\end{cases}
	\end{align*}
with $\bar u(0,x)=u_0(x)$ and $\underline v(0,x)=v_0(x)$ for $x\in \R$. Then from the comparison principle, see Lemma \ref{lemmacomp2}, we obtain
\begin{align*}
	u(t,x)\leq  \bar  u(t,x),\ v(t,x)\geq  \underline  v(t,x)\ \mbox{ for } \ t\geq 0,\ x\in \R.
\end{align*}
Following a well known iteration argument (see, for example, the proof  of Theorem 3.2 in \cite{Zwy-2022-dcds-b}) one can show that 
\begin{align*}
	\lim\limits_{t\rightarrow\infty} \bar   u(t, x) = 0, ~\lim\limits_{t\rightarrow\infty}\underline  v(t, x) = 1 {\rm ~locally~ uniformly~ in~} \mathbb{R}.
\end{align*}
This clearly implies \eqref{0-1}. Moreover,   there is $t_1>0$ such that 
\begin{align*}
	v(t,x)\geq \underline  v(t, x) \geq  \frac{1+k}{2k}\ \mbox{ for } \ t\geq t_1, \ x\in [-2h_0,2h_0]. 
\end{align*}
Note that $kv(t,x)\geq \frac{1+k}2>1$ for $t\geq t_1$ and  $x\in [-2h_0,2h_0]$. Then, as in Lemma \ref{l4.3} (iii), we could construct an upper solution for $(u,g,h)$ and show that $h_\infty - g_\infty < \infty$  for sufficient small $\mu$. The desired conclusion then follows from the comparison principle. 
\end{proof}

 Clearly Theorem \ref{v-1} follows directly from Lemma \ref{lem:u-inf}, and  Theorem \ref{th1.1} follows from Lemmas \ref{lem:mu_*}, \ref{l4.3} and \ref{lem:u-inf}.

\section{Appendix: some comparison principles}\label{PS}

Let $(u,v,g,h)$ be the solution of \eqref{KnK1.2}. For  $T>0$, $g_1, h_1\in C([0,T])$  with  $ g_1(t)<  h_1(t)$, we will use the notation
 \[
  [0,T]\times [g_1,h_1]:=\{(t,x): t\in [0,T],\ x\in [g_1(t),h_1(t)]\}.
  \]
 Let $m\in C([0,T]\times \R)$ be a bounded function and 
\begin{align*}
	F(t,x,s):=s(m(t,x)-s) \ {\rm for}\ s\geq 0,\ t\geq 0, \ x\in \R.
\end{align*}

We list below several comparison principles, whose proofs are easily obtained by following the proof in \cite{Cao-2019-JFA}, and are omitted here.

\begin{lemma}[Comparison principle 1]\label{lemmacomp1} Assume that $ \bar g,\,\bar h, \, \underline  g,\,\underline  h\in C([0,T])$ satisfy $\underline g(t)<\underline  h(t)$ and $\bar g(t)<\bar  h(t)$; the functions $\bar u\in C([0,T]\times \R)\cap C^{1,0}([0,T]\times [\bar g,\bar h])$ and $\underline u\in C([0,T]\times \R)\cap C^{1,0}([0,T]\times [\underline g,\underline h])$  are nonnegative and bounded. 
	
	{\rm (i)} {\rm \underline {(Two sides free boundaries)}} Suppose $\bar g'\leq 0 \leq \bar h'$ and $\underline g'\leq 0 \leq \underline h'$. If  $(\bar u,  \bar g,\bar h)$   satisfies 	
	\begin{align}\label{compsing1}
		\begin{cases}
			\dd	\bar u_t \geq  d_1\int_{\bar g(t)}^{\bar h(t)}J_1(x - y)\bar u(t, y)\rd y - d_1\bar u + F(t,x,\bar u),   &
			(t,x) \in  (0,T] \times (\bar g,\bar h),\\[4mm]
			\bar u(t, x) \geq   0,  &
			t \in [0,T],\ x \not\in (\bar g(t), \bar h(t)),\\
			\dd	\bar h^{\prime} (t) \geq \mu\int_{\bar g(t)}^{\bar h(t)}\int_{\bar h(t)}^{\infty}J_1(x - y)\bar u(t, x)\rd y\rd x,  &
			t \in (0,T],\\[4mm]
			\dd	\bar g^{\prime} (t) \leq  -{\mu}\int_{\bar g(t)}^{\bar h(t)}\int_{-\infty}^{\bar g(t)}J_1(x - y)\bar u(t, x)\rd y\rd x,  &
			t \in (0,T],
		\end{cases}
	\end{align}
	and $(\underline  u,  \underline g,\underline h)$ satisfies  \eqref{compsing1} with all the inequality signs reversed,  and 
	\[
	[\underline g(0),\underline h(0)]\subset [\bar g(0),\bar h(0)],\ 
	\bar u(0, x) \geq \underline u(0,x) \mbox{ for }
	x \in [\underline g(0),\underline h(0)],
	\]
	then  
	\[
	 \begin{cases}\underline  g(t),\underline h(t)]\subset [\bar g(t),\bar h(t)]\mbox{ for  $ t\in [0,T]$,} \\
		\underline u(t,x)\leq  u(t,x) \mbox{ for } \ t\in [0,T],\  x\in [\underline  g(t),\underline h(t)].
		\end{cases}
	\]

	{\rm (ii)} {\rm \underline {(One side free boundary)}} Suppose $\bar g<  \bar h$, $\bar g'\leq 0 \leq \bar h'$ and $(\bar u,  \bar g,\bar h)$   satisfies 	\eqref{compsing1}. If $\bar g(t)\leq \underline g(t)<\underline h(t)$ and $\underline h'(t)\geq 0$ for $t\in (0, T]$, and
	$(\underline u,  \underline g,\underline h)$   satisfies 	
	\begin{align*}
		\begin{cases}
			\dd	\underline u_t \leq  d_1\int_{\underline g(t)}^{\underline h(t)}J_1(x - y)\underline u(t, y)\rd y - d_1\underline u + F(t,x,\underline u),   &
			(t,x) \in  (0,T] \times (\underline g,\underline h),\\[4mm]
			\underline u(t, x) =   0,  &
			t \in [0,T],\ x \in \{ \underline g(t), \underline h(t)\},\\
			\dd	\underline h^{\prime} (t) \leq \mu\int_{\underline g(t)}^{\underline h(t)}\int_{\underline h(t)}^{\infty}J_1(x - y)\underline u(t, x)\rd y\rd x,  &
			t \in (0,T],\\
			\underline h(0)\leq \bar h(0),\ \underline u(0,x)\leq \bar  u(0,x),& x\in [\underline g(0),\underline h(0)],
		\end{cases}
	\end{align*}
	then 
	\begin{align*}
		\underline h(t)\leq \bar h(t),\ \ \ \underline u(t,x)\leq \bar  u(t,x)\ {\rm for}\ t\in [0,T]\ {\rm and}\ x\in  [\underline g(t),\underline h(t)].
	\end{align*}
	
	{\rm (iii)} {\rm \underline {(One side free boundary)}} Suppose $\underline g<  \underline h$, $\underline g'\leq 0 \leq \underline h'$ and $(\underline u,  \underline g,\underline h)$   satisfies 	\eqref{compsing1} with all the inequality signs reversed.  If $h_1\in C([0,T])$ satisfies  $h_1(t)\leq \underline h(t)$ for $t\in [0, T]$,  and $(\bar u,  \bar g,\bar h)$   satisfies 	
	\begin{align*}
		\begin{cases}
			\dd	\bar u_t \geq  d_1\int_{\bar g(t)}^{\yy}J_1(x - y)\bar u(t, y)\rd y - d_1\bar u + F(t,x,\bar u),   &
			(t,x) \in  (0,T] \times (\bar g,\yy),\\[4mm]
			\bar u(t, x) \geq \underline u(t, x) \geq    0,  &
			t \in [0,T],\ x \in\{ \bar g(t)\}\cup [ h_1(t),\yy)\\
			\dd	\bar g^{\prime} (t) \leq  -{\mu}\int_{\bar g(t)}^{\yy}\int_{-\infty}^{\bar g(t)}J_1(x - y)\bar u(t, x)\rd y\rd x,  &
			t \in (0,T],\\
			\bar g(0)\leq 	\underline g(0),\ 	\bar u(0,x)\geq \underline  u(0,x),& x\in [\bar g(0),\yy)
		\end{cases}
	\end{align*}
	then 
	\begin{align*}
		\bar g(t)\leq \underline g(t)\ {\rm for}\ t\geq 0,\ \ \ \underline u(t,x)\leq \bar  u(t,x)\ {\rm for}\ t\in [0,T]\ {\rm and}\ x\in  [\bar g(t),\yy).
	\end{align*}
	
\end{lemma}

\begin{lemma}[Comparison principle 2]\label{lemmacomp2} Assume that $ \bar g,\,\bar h, \, \underline  g,\,\underline  h\in C([0,T])$ satisfy $\underline g(t)<\underline  h(t)$ and $\bar g(t)<\bar  h(t)$, and have the monotone properties $\bar g'\leq 0$, $\underline g'\leq 0$, $\bar h'\geq  0$ and $\underline h'\geq  0$. Suppose that  
	\begin{align*}
		\bar u\in C([0,T]\times \R)\cap C^{1,0}([0,T]\times [\bar g,\bar h]),\ \underline v\in C([0,T]\times \R)\cap C^{1,0}([0,T]\times [\underline g,\underline h]),
	\end{align*}
	and $\bar u,\, \underline v\in C^{1,0}([0,T]\times \R)$ are nonnegative and bounded. 
	
	{\rm (i)} {\rm \underline {(Two sides free boundaries)}} If  $(\bar u, \underline v, \bar g,\bar h)$   satisfies 	
	\begin{align}\label{comp1}
		\begin{cases}
			\dd	\bar u_t \geq  d_1\int_{\bar g(t)}^{\bar h(t)}J_1(x - y)\bar u(t, y)\rd y - d_1\bar u + \bar u(1 - \bar u - k\underline v),   &
			(t,x) \in  (0,T] \times (\bar g,\bar h),\\[4mm]
			\dd			\underline v_t \leq   d_2\int_\mathbb{R}J_2(x - y)	\underline v(t, y)\rd y - d_2	\underline v + \gamma 	\underline v(1 - 	\underline v - h\bar  u),   &
			(t,x) \in  (0,T] \times \mathbb{R},\\[4mm]
			\bar u(t, x) \geq   0,  &
			t \in [0,T],\ x \not\in (\bar g(t), \bar h(t)),\\
			\dd	\bar h^{\prime} (t) \geq \mu\int_{\bar g(t)}^{\bar h(t)}\int_{\bar h(t)}^{\infty}J_1(x - y)\bar u(t, x)\rd y\rd x,  &
			t \in (0,T],\\[4mm]
			\dd	\bar g^{\prime} (t) \leq  -{\mu}\int_{\bar g(t)}^{\bar h(t)}\int_{-\infty}^{\bar g(t)}J_1(x - y)\bar u(t, x)\rd y\rd x,  &
			t \in (0,T],\\	
			\bar u(0, x) \geq u_0(x),\ \underline v(0, x) \leq  v_0(x),
			&x \in\mathbb{R},
		\end{cases}
	\end{align}
	and $\bar u(0,x)\geq 0$ for $x\in \R$,  then  the unique solution $(u,v,g,h)$ of \eqref{KnK1.2} satisfies
	\begin{align*}
		&[g(t),h(t)]\subset [\bar g(t),\bar h(t)],\ \  t\in [0,T],\\
		&u(t,x)\leq \bar u(t,x),\ \underline v(t,x)\leq  v(t,x),\ \ \ t\in [0,T],\  x\in \R.
	\end{align*}
		If $(\underline  u, \bar v, \underline g,\underline h)$ satisfies  \eqref{comp1} with all the inequality signs reversed,   then  
\begin{align*}
		& [\underline  g(t),\underline h(t)]\subset [g(t),h(t)], \ \  t\in [0,T],\\
		&\underline u(t,x)\leq  u(t,x),\ v(t,x)\leq \bar v(t,x),\ \ \ t\in [0,T],\  x\in \R.
	\end{align*}
	
	{\rm (ii)} {\rm \underline {(Fixed  boundaries)}} Suppose $g(t)<h(t)$, $g'(t)\leq 0\leq h'(t)$, and $A\in C([0,T]\times [g,h]$ is a nonnegative function. Assume that   $\bar u, \underline v$   are nonnegative and satisfy 	
	\begin{align}\label{comp2}
		\begin{cases}
			\dd	\bar u_t \geq  d_1\int_{ g(t)}^{ h(t)}J_1(x - y)\bar u(t, y)\rd y - d_1\bar u + \bar u(1 - \bar u - k\underline v),   &
			(t,x) \in  (0,T] \times (g, h),\\[4mm]
			\dd			\underline v_t \leq   d_2\int_{ g(t)}^{ h(t)}J_2(x - y)	\underline v(t, y)\rd y - d_2	\underline v +A+ \gamma 	\underline v(1 - 	\underline v - h\bar  u),   &
			(t,x) \in  (0,T] \times ( g, h),
		\end{cases}
	\end{align}
	and $\underline  u, \bar v$ are nonnegative and satisfy  \eqref{comp2} with all the inequality signs reversed. If 
	\begin{align}\label{comp3}
		\begin{cases}
			\bar u(t, x) \geq \bar v(t, x)\geq   0,  &
			t \in [0,T],\ x \in\{g(t), h(t)\},\\
			\bar u(0, x) \geq \underline u(0,x),\ \bar v(0, x) \geq  \underline v(0,x),
			&x \in [ g(0), h(0)],
		\end{cases}
	\end{align}
	then 
	\begin{align*}
		\bar u(t, x) \geq \underline u(t,x),\ \bar v(t, x) \geq  \underline v(t,x)\ \mbox{ for }\ t\in [0,T],\ 
		x \in [ g(t), h(t)].
	\end{align*}
	We remark that the condition in  first inequality of \eqref{comp3} can be removed if $ g$ and $ h$ are constants. 
	
	{\rm (iii)} {\rm \underline {(Comparison with the Cauchy problem)}} Let $(u,v,g,h)$ be the solution of \eqref{KnK1.2}. If   $\bar u, \underline v$  are nonnegative and   satisfy 	
	\begin{align*}
		\begin{cases}
			\dd	\bar u_t \geq  d_1\int_{\R}J_1(x - y)\bar u(t, y)\rd y - d_1\bar u + \bar u(1 - \bar u - k\underline v),   &
			(t,x) \in  (0,T] \times \mathbb{R},\\[4mm]
			\dd			\underline v_t \leq   d_2\int_\mathbb{R}J_2(x - y)	\underline v(t, y)\rd y - d_2	\underline v + \gamma 	\underline v(1 - 	\underline v - h\bar  u),   &
			(t,x) \in  (0,T] \times \mathbb{R},\\[4mm]
			\bar u(0, x) \geq u_0(x),\ 0\leq \underline v(0, x) \leq  v_0(x),
			&x \in\mathbb{R},
		\end{cases}
	\end{align*}
	then $u(t,x)\leq \bar u(t,x)$ and $ v(t,x)\geq \underline v(t,x)$ for $t\in [0,T]$ and $x\in \R.$
\end{lemma}
\begin{corollary}\label{coro4.4}
	The solution $u$ is nondecreasing  with respect to $\mu$, $h$ and $u_0$, while nonincreasing  with respect to  $k$ and $v_0$.
\end{corollary}

\end{document}